\documentclass{amsart}

\usepackage{psfrag}

\usepackage{setspace}
\usepackage[T1]{fontenc}
\usepackage[dvips]{graphicx}
\usepackage{epsfig}
\usepackage{amsmath,amssymb,amscd,amsthm}
\usepackage{subfigure}
\usepackage[latin1]{inputenc}
\usepackage{latexsym}
\usepackage{graphics}
\usepackage{colortbl}
\usepackage{color}
\usepackage{ae}
\usepackage{enumitem}

\usepackage{textcomp}
\begin{document}

\setcounter{secnumdepth}{3}
\setcounter{tocdepth}{2}

\newtheorem{definition}{Definition}[section]
\newtheorem{lemma}[definition]{Lemma}
\newtheorem{sublemma}[definition]{Sublemma}
\newtheorem{corollary}[definition]{Corollary}
\newtheorem{proposition}[definition]{Proposition}
\newtheorem{theorem}[definition]{Theorem}
\newtheorem{fact}[definition]{Fact}
\newtheorem{question}[definition]{Question}
\newtheorem{remark}[definition]{Remark}
\newtheorem{example}[definition]{Example}
\newtheorem{assumption}[definition]{Assumption}

\newcommand{\cov}{\mathrm{covol}}
\def \tr{{\mathrm{tr}}}
\def \det{{\mathrm{det}\;}}
\def\co{\colon\tanhinspace}
\def\I{{\mathcal I}}
\def\N{{\mathbb N}}
\def\R{{\mathbb R}}
\def\Z{{\mathbb Z}}
\def\Sph{{\mathbb S}}
\def\Tor{{\mathbb T}}
\def\Disk{{\mathbb D}}
\def\Hess{\mathrm{Hess}}
\def\rad{\mathbf{v}}

\def\H{{\mathbb H}}
\def\RP{{\mathbb R}{\mathrm{P}}}
\def\dS{{\mathrm d}{\mathbb{S}}}
\def\Isom{\mathrm{Isom}}

\def\pr{\mathrm{pr}}

\def\sh{\mathrm{sinh}\,}
\def\ch{\mathrm{cosh}\,}
\newcommand{\arccosh}{\mathop{\mathrm{arccosh}}\nolimits}
\newcommand{\oh}{\overline{h}}

\newcommand{\mf}{\mathfrak}
\newcommand{\mb}{\mathbb}
\newcommand{\ol}{\overline}
\newcommand{\la}{\langle}
\newcommand{\ra}{\rangle}
\newcommand{\hess}{\mathrm{Hess}\;}
\newcommand{\grad}{\mathrm{grad}}
\newcommand{\M}{\mathrm{MA}}
\newcommand{\II}{\textsc{I\hspace{-0.05 cm}I}}
\renewcommand{\d}{\mathrm{d}}
\newcommand{\A}{\mathrm{A}}
\renewcommand{\L}{\mathcal{L}}
\newcommand{\FB}[1]{{\color{red}#1}}
\newcommand{\note}[1]{{\color{blue}{\small #1}}}
\newcommand{\Area}{\mathrm{Area}}

\setlength{\abovedisplayshortskip}{1pt}
\setlength{\belowdisplayshortskip}{3pt}
\setlength{\abovedisplayskip}{3pt}
\setlength{\belowdisplayskip}{3pt}


\title[Embeddings in flat Lorentzian manifolds]{Embeddings of non-positively curved compact surfaces in flat Lorentzian manifolds}

\author{Fran\c{c}ois Fillastre, Dmitriy Slutskiy}

\date{ \today \\ V2}

\begin{abstract}
We prove that any metric of non-positive curvature in the sense of Alexandrov on a compact surface can be isometrically embedded as a convex spacelike Cauchy surface in a flat  spacetime of dimension (2+1).
The proof follows from polyhedral approximation.
\end{abstract}

\maketitle



\setcounter{tocdepth}{3}
\tableofcontents

\section{Introduction}

In the 1940's, A.D. Alexandrov, looking at the induced (intrinsic) distances on the boundary of convex bodies of the Euclidean space, introduced a class of distances on compact surfaces. Nowadays, such distances are called \emph{metrics of non-negative curvature (in the sense of Alexandrov)}. He then proved the following famous result \cite{alex}. We assume that all the surfaces we are considering are closed, oriented and connected.

\begin{theorem}\label{thm:alex}
Let $(S,d)$ be a metric of non-negative curvature
 on a compact surface.
 Then there exists a flat  Riemannian manifold $R$ homeomorphic to $S\times \R$
 which contains a  convex surface whose induced distance is isometric to $(S,d)$.
\end{theorem}

Actually, if $(S,d)$ is a metric space isometric to the induced distance on a flat torus $(T,h)$, then the statement above is trivial, as $R$ can be taken as $S\times\R$ with the metric $h+\operatorname{d}t^2$. Otherwise, by the Gauss--Bonnet formula, a compact surface  $S$ with a metric of non-negative curvature  must have genus $0$. In this case,  the metric
on $R$ is the one of the Euclidean space minus the origin.
A more classical way to state Theorem~\ref{thm:alex} in this case
is to say that $(S,d)$ is isometric to the induced distance on the boundary of a convex body of the Euclidean space.

In the present paper, we prove an analogous result for metrics of non-positive curvature.

\begin{theorem}\label{thm1}
Let $(S,d)$ be a metric of non-positive curvature (in the sense of Alexandrov) on a compact surface.
 Then there exists a flat  Lorentzian manifold $L$ homeomorphic to $S\times \R$
 which contains a spacelike convex surface whose induced distance is isometric to $(S,d)$.

\end{theorem}

The definition of metric of non-positive curvature is recalled in Section~\ref{paragraph:construction-of-polyhedral-metrics}.
Once again, if $(S,d)$ is a metric space isometric to the induced distance on a flat torus $(T,h)$,
then the statement above is trivial, as $L$ can be taken as $S\times\R$ with the metric $h-\operatorname{d}t^2$.
 Otherwise, by Gauss--Bonnet formula, $S$ must have genus $g>1$. We will now consider only this case.

Theorem~\ref{thm:alex} had multiple generalizations. One can isometrically embed metric spaces with curvature bounded from below on compact surfaces into constant curvature Riemannian $3$d-spaces. See the introduction of \cite{fiv} for an overview. Also one can try to  isometrically embed metric spaces with curvature bounded from \emph{above} on compact surfaces into constant curvature \emph{Lorentzian} $3$d-spaces. The first results in this direction were proved in \cite{Riv86}, and \cite{RH1993,Sch1996}. Other particular results  are cited below. As far as we know, Theorem~\ref{thm1} is the first result in this direction where no restriction ---apart form the curvature bound, is imposed to the metric (e.g. to come from a smooth Riemannian metric, or to be a  constant curvature metric with conical singularities).

We will state a Theorem~\ref{thm:main 2} below, which will imply Theorem~\ref{thm1}. But let us give some definitions before stating Theorem~\ref{thm:main 2}.

The Minkowski space $\R^{2,1}$  is
$\R^{3}$ endowed with the bilinear form
$$\langle x,y\rangle_-=x_1y_1+x_2y_2-x_{3}y_{3}~.$$
A plane $P$ is \emph{spacelike} if the restriction of
$\langle \cdot,\cdot\rangle_-$ to $P$ is positive definite.
A \emph{spacelike convex set} $K$ of $\R^{2,1}$ is a closed convex set which has only spacelike planes as support planes. This assumption is more restrictive than only asking that $K$ is the intersection of half-spaces bounded by spacelike planes, as shows the closure of
$$I^+(0)=\{x \in \R^3 | \langle x,x\rangle_{-} <0,\,x_3>0\}~.$$
Without loss of generality, we can assume that the set $K$ is \emph{future convex}, i.e. $K$ is the intersection of the future sides of its support planes (the future side of the plane is the one containing the vector $(0,0,1)$).
A \emph{spacelike convex surface}  is the boundary of a (future) convex spacelike set of Minkowski space.

The \emph{induced distance}
on a spacelike convex surface is the distance induced by the length structure given by
\begin{equation}\label{eq:length a}\mathfrak{L}(c)=\int_0^1 \| c'\|_-~,\end{equation}
where, for a spacelike vector $v$,
$$\|v\|_- = \langle v,v\rangle_-^{1/2}~, $$
and $c :[0,1]\rightarrow \partial K$
a Lipschitz curve (with respect to the ambient Euclidean metric of $\R^3$). Note that as the set is convex, there is always a Lipschitz curve between two points on the boundary.

A famous example is $\{x\in I^+(0)| \langle x,x\rangle_-<-1\}$
whose  boundary  is

\begin{equation}\label{def:upper-hyperboloid}
\H^2=\{x\in I^+(0)| \langle x,x\rangle_-=-1\}~.
\end{equation}

If $d_{\H^2}$ is the induced distance
on $\H^2$, then $(\H^2,d_{\H^2})$ is isometric to the distance of the hyperbolic plane. Moreover, this implies a canonical identification between the group of orientation-preserving isometries of the hyperbolic plane and the connected component of the identity  $O_0(1, 2)$ of  $O(1,2)$.

Let us go back to the distance $(S,d)$ of the statement of Theorem~\ref{thm1} and to its universal cover $( \tilde S,\tilde d)$ (see Section~\ref{sec:met on univ cov} for the definition of the distance $\tilde d$).

We want to find a \emph{spacelike convex isometric immersion} of $(\tilde S,\tilde d)$ into $\R^{2,1}$,  i.e.,
a map
$$\phi : \tilde S \to \R^{2,1} $$
such that
\begin{itemize}
\item $\phi(\tilde S)$ is a spacelike convex surface
\item $\phi$ is an isometry between $(\tilde S,\tilde d)$ and
$\phi(\tilde S)$ endowed with the induced distance.
\end{itemize}
Moreover we want $\phi$ to be \emph{equivariant}, that is, there exists a faithful and discrete representation
$$\rho : \pi_1S\to \mathrm{Isom}\, \R^{2,1} $$
such that for all
$\gamma\in \pi_1S$ and for all $x\in \tilde S$,
\begin{equation}\label{eq:equivariance}\phi(\gamma. x)=\rho(\gamma)\phi(x)~, \end{equation}
where the action of $\pi_1S$ onto $\tilde S$ is by deck transformations.

It follows that the distance induced on $\phi(\tilde S)/\rho(\pi_1S)$ is isometric to $(S,d)$.
In the present paper, we will look more precisely for
a spacelike convex isometric immersion which is a \emph{Fuchsian convex isometric immersion}. This means that
\begin{itemize}
\item $\phi(\tilde S)$ is contained in $I^+(0)$ (recall that our convex sets are implicitly future convex),
\item  $\rho(\pi_1S)$
is a Fuchsian subgroup of $O_0(2,1)$.
\end{itemize}
One can easily see that in the  Fuchsian case,
$\phi(\tilde S)$ meets exactly once each future timelike half-line.

Let us give a trivial example. If $(S,d)$ is the distance given by a hyperbolic metric $h$ on $S$, then there is an isometry between
$(\tilde S,\tilde d)$ and $\H^2$ (the developing map of the hyperbolic structure given by $h$) and there
exists an equivariant representation $\rho:\pi_1S\to O_0(2,1)$ (the holonomy)  such that $\H^2/\rho(\pi_1S)$ with its induced distance is isometric to $(S,d)$.

In the present paper, we prove the following theorem.
\begin{theorem}\label{thm:main 2}
Let $(S,d)$ be a metric of non-positive curvature on a compact surface (of genus $\geq 2$).
Then there exists a Fuchsian convex isometric  immersion
of  $(\tilde S,\tilde d)$ into Minkowski space.
\end{theorem}

Theorem~\ref{thm:main 2} implies Theorem~\ref{thm1}: the surface $\phi(\tilde S)/\rho(\pi_1S)$ in the flat Lorentzian manifold $L=I^+(0)/\rho(\pi_1S)$ is isometric to $(S,d)$.
One can moreover precise the statement of Theorem~\ref{thm1}:
\begin{itemize}
\item the flat Lorentzian manifold $L$ contains a totally umbilic hyperbolic surface, namely, $\H^2/\rho(\pi_1S)$;
\item actually, if $g$ is the Riemannian metric of $\H^2/\rho(\pi_1S)$, the metric of $L$ is $t^2g-\mbox{d}t^2$;
\item the isometric embedding of $(S,d)$ is a \emph{Cauchy surface}, i.e. it meets exactly once each inextensible non spacelike curve of $L$.
\end{itemize}

Theorem~\ref{thm:main 2} is already known when $(S,d)$ comes from a smooth Riemannian metric $h$ on $S$:
\begin{theorem}[{\cite{LS00}}]\label{thm:ls}
Let $(S,h)$ be a smooth Riemannian metric of negative sectional curvature on a compact surface.
Then there exists a smooth Fuchsian convex isometric  immersion
of $(\tilde S,\tilde h)$ into Minkowski space.
\end{theorem}

Moreover, it is proved in \cite{LS00} that the  immersion is unique among smooth Fuchsian convex immersions, up to composition by an element of $O(2,1)$. If the image of the immersion given by Theorem~\ref{thm:main 2} is known to be $C^1$, then finer results about the regularity of the immersion with respect to the regularity of the metric are available \cite{sok77}. There also exists a result about isometric immersions into $I^+(0)$ of smooth metrics on the disc of curvature with a negative upper bound  \cite{chen-yin}.

Theorem~\ref{thm:main 2} is also already known when $(S,d)$ is a \emph{polyhedral} metric of non-positive curvature, which means that $(S,d)$ is a flat metric on $S$ with conical singularities of negative curvature (i.e. the cone angles are $>2\pi$ at the singular points).
Theorem~\ref{thm:cas poly} below was proved in \cite{Fil11} using a deformation method. More recently, a generalization of Theorem~\ref{thm:cas poly}  was proved  in \cite{leo}, using a variational method.
\begin{theorem}[{}]\label{thm:cas poly}
Let $(S,d)$ be a polyhedral metric of non-positive curvature.
Then there exists a polyhedral Fuchsian convex isometric  immersion
of $(\tilde S,\tilde d)$ into Minkowski space.
\end{theorem}

An immersion is \emph{polyhedral} if its image is the boundary of the convex hull (in $\R^3$) of the orbit for $\rho(\pi_1S)$ of a finite number of points in $I^+(0)$. It implies in particular that it is a gluing of compact convex Euclidean polygons \cite{Fil12}.
Here also, the Fuchsian immersion is unique (among polyhedral Fuchsian convex  immersions) up to global isometries. Uniqueness is not known in the general case of Theorem~\ref{thm:main 2}.

The proof of Theorem~\ref{thm:main 2} will be by a classical polyhedral approximation, using Theorem~\ref{thm:cas poly}. Hence we will need to prove some
convergence and compactness results. Note that in Minkowski space, things may behave very differently than in the classical Euclidean space. Major differences may be summarized as follows.
\begin{itemize}
\item There is no Busemann--Feller lemma in Minkowski space. This lemma says that the orthogonal projection onto a convex set does not increase the lengths.
\item The preceding fact is a consequence of the fact that there is no triangle inequality in Minkowski space. Instead, the reversed Cauchy--Schwarz inequality holds: for any vectors $u$ and $v$ in a Minkowski plane
$$\langle u,v \rangle_-^2  \geq \langle u,u\rangle_-\langle v,v\rangle_-~. $$
If moreover $u$ and $v$ are spacelike, then $\langle u,u \rangle_-$ is positive, and the inequality above
leads to the reversed triangle inequality
\begin{equation}\label{reversed ti} \|u\|_- + \|v\|_- \leq \|u+v\|_-~,\end{equation}
i.e. if $x,y,z$ are three points in a Minkowki plane, related by space-like segments, then the Minkowski distance between $x$ and $y$ is greater than the sum of the distance between $x$ and $z$ and the distance between $z$ and $y$.
\item There is no Blaschke selection theorem. This result says that if a sequence of convex surfaces passes through some common point, it suffices to have an uniform bound on the diameters of the distance induced by the ambient Euclidean metric to have a converging subsequence of surfaces.
\item The length structure given by \eqref{eq:length a}
induces a distance $d$, which itself gives a length structure $L_d$. It is not obvious that both length structure coincide on the set of Lipschitz curves. Also note that a priori, $d$ is only a pseudo-distance (i.e. the distance of two distinct points may be zero).
\end{itemize}

Concerning the first point, there is a kind of analogue in Minkowski space, that roughly says that the orthogonal  projection from the \emph{past} of the convex set (which may be an empty set) onto the convex set expands the lengths, see \cite[6.1]{BBZ} for a precise statement. However in the present paper we will need only a trivial case (Lemma~\ref{lem: comp hyp met1}).

Another issue that appears in our case, is that we are looking at surfaces in Minkowski space which are invariant under the action of a group of isometries, and, given a sequence of equivariant immersions, the groups and the immersions may both degenerate, in such a way that the sequence of induced distances  converges. Figure~\ref{fig1} shows this situation in the case of the Minkowski plane: there is an isometry $I_t$
of Minkowski plane, corresponding to a hyperbolic translation of length $t$, and points $\epsilon x$ and $\epsilon I_t(x)$ at (Minkowski) distance $\epsilon$ from the origin. Then it may happen that $t\to\infty$ and $\epsilon \to 0$, but the Minkowski length of the spacelike segment  between $\epsilon x$ and $\epsilon I_t(x)$ remains constant.

\newsavebox{\smlmat}
\savebox{\smlmat}{$\left(\begin{smallmatrix}\cosh(t)&\sinh(t)\\\sinh(t)&\cosh(t)\end{smallmatrix}\right)$}

\begin{figure}[h]\begin{center}
\includegraphics[scale=0.4]{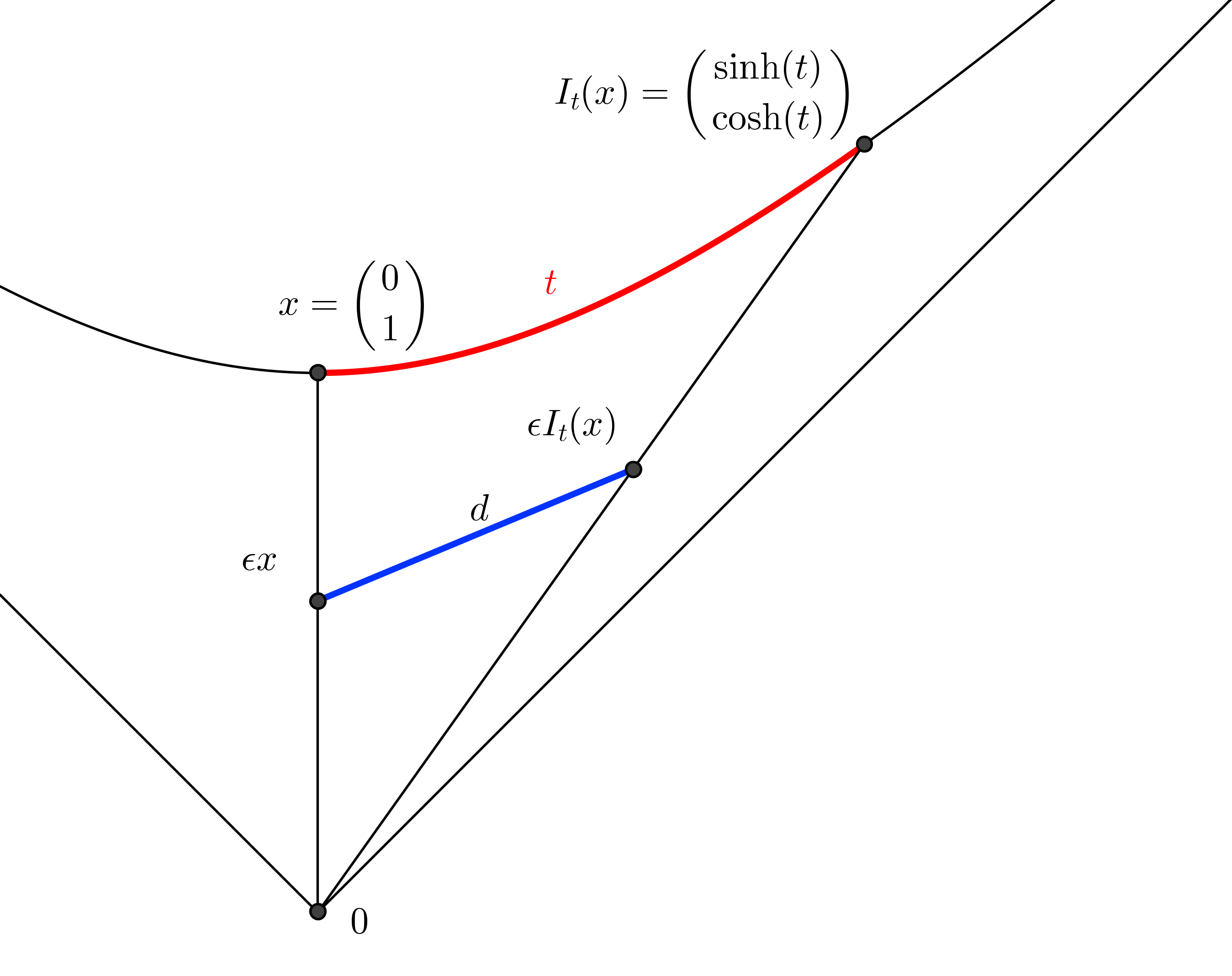}
\end{center}\label{fig1}\caption{We easily compute that $\epsilon\left(\cosh(t)-1\right)^{1/2}=d/\sqrt{2}$, so there is a suitable $\epsilon$ such that even if $t$ is arbitrary large or close to $0$, the length $d$ between $\epsilon x$ and $\epsilon I_t(x)$ remains constant, with $I_t$ the isometry of the Minkowski plane represented by~\usebox{\smlmat}.}\end{figure}

%

In the setting of Theorem~\ref{thm:main 2},
as we are dealing with groups acting cocompactly on $\H^2$,
the lengths of all the hyperbolic translations cannot be arbitrarily large.
This point is formalized in Section~\ref{section:radial-functions}.

The paper is organized as follows. In Section~\ref{sec 2},
 we recall basic facts about uniform convergence of metric spaces, as well as
 a Theorem of A.D.~Alexandrov and V.A.~Zalgaller about triangulation of surfaces.
In Section~\ref{sec 3}, roughly speaking, we prove that if the images of a sequence of polyhedral Fuchsian convex isometric immersions are contained between two hyperboloids, then there is a converging subsequence of immersions, and the induced distances are converging too.
In Section~\ref{sec 4} we show that if
the induced distances of a sequence of polyhedral Fuchsian convex isometric immersions  converge, then there is a subsequence of converging  surfaces. Section~\ref{sec 4}
is the main part of the paper. Eventually, in Section~\ref{sec 5}, all the elements are put together to provide a proof of Theorem \ref{thm:main 2}.

\medskip

\textbf{Acknowledgment.} The authors want to thank Giona Veronelli who pointed out an error in a preceding version of this text.
Most of this work was achieved when the second author was a post-doc in the AGM institute of the Cergy--Pontoise University. He wants to thank the institution for its support.

\section{Uniform convergence of metric spaces}\label{sec 2}

We first recall some very basic facts about uniform convergence, for which we were unable to find a reference. Then in Section~\ref{paragraph:construction-of-polyhedral-metrics}
we recall a theorem of Alexandrov and Zalgaller about
polyhedral approximations of particular distances on surfaces, which is the cornerstone of the proof of our main theorem.

In the sequel, we will denote $\N  \cup \{\infty\}$ by $\bar \N$. When we will say
``$d$ is a distance on the manifold $M$'', we imply that   the topology induced by $d$ is the topology of the manifold $M$.

\subsection{Uniform convergence}

Recall that a sequence of metric spaces $(M_n,d_n)_n$ uniformly converges to the metric space
$(M,d)$ if there exist homeomorphisms $f_n:M_n\to M$ such that
$$\sup_{x,y\in M_n}|d(f_n(x),f_n(y)) - d_n(x,y) | $$
goes to $0$ when $n$ goes to infinity. If $M_n=M$ for all $n$, then the definition is the usual definition of uniform convergence, considering the distances as maps from $M\times M$ to $\R$.
This is a more restrictive notion of convergence than the usual Gromov--Hausdorff convergence. Actually, uniform convergence implies Gromov--Hausdorff convergence,  see e.g. \cite{BBI2001}. But uniform convergence is the suitable notion for our needs.

The following trick is maybe due to Alexandrov \cite{alex}.

\begin{lemma}\label{lem alex trick}
Suppose that distances $(d_n)$ pointwise converge to $d_\infty$ on a compact manifold $M$, and that there is a distance $d_{\operatorname{max}}$ such that for any $n\in \bar \N$,
$d_n \leq d_{\operatorname{max}}. $ Then the convergence is uniform.
\end{lemma}
\begin{proof}
By the triangle inequality,
 $$|d_{n}(x,y)-d_{n}(p,q)|\leq d_{n}(x,p)+d_{n}(y,q) $$ $$\leq  d_{\operatorname{max}}(x,p) + d_{\operatorname{max}}(y,q)~.$$

So  the family of continuous functions $d_{n}$ on $M\times M$ endowed with the product distance $d_{\operatorname{max}} + d_{\operatorname{max}}$ is equi-Lipschitz, hence equicontinous. By Arzela--Ascoli theorem, the convergence of  $(d_{n})_n$  is uniform.
\end{proof}

\begin{lemma}\label{sup metrique}
Let $M$ be a  compact manifold, and
$d_n,d_\infty$ distances on $M$, such that
 $d_n$ uniformly converge to $d_\infty$.
The map $d_{\sup}:M\times M \to \R$ defined by
$$d_{\sup}(x,y)= \sup_{n\in \bar \N} d_n(x,y) $$
is a distance on $M$.
\end{lemma}
Actually it is straightforward that $d_{\sup}$ is a distance. What we imply in the statement of this lemma is that the topology induced by $d_{\sup}$ is the same as the topology of $M$.
It relies on the following fact.

\begin{fact}\label{sup continu}
Let $(f_n)_{n\in \N}$ be a sequence of continuous functions on a compact metric space $(E,m)$, uniformly converging to a function $f_\infty$. Then the function $g=\sup_{n\in \bar \N } f_n$ is continuous.
\end{fact}
\begin{proof}
Let $\epsilon >0$.
The function $f_\infty$ is continuous hence uniformly continuous on $E$.
So there exists $\delta_\infty>0$ such that if $m(x,y)<\delta_\infty$, then
$$|f_\infty(x)-f_\infty(y)|\leq \epsilon/3~. $$
Moreover, by uniform convergence, there exists $N$ such that for any $n\geq N$ and any $x\in E$,
$$|f_n(x)-f_\infty(x)|\leq \epsilon/3$$
so for $n\geq N$ and $x,y$ such that $m(x,y)<\delta_\infty$,
$$|f_n(x)-f_n(y)| \leq |f_n(x)-f_\infty(x)|+|f_\infty(x)-f_\infty(y)|+|f_\infty(y)-f_n(y)|\leq \epsilon~. $$
Also, for any $ n \leq N$, the function $f_n$ is uniformly continuous, hence there exists $\delta_n>0$ such that if $m(x,y)<\delta_n$ then
$$|f_n(x)-f_n(y)|\leq \epsilon~. $$
So, for $\delta=\min \{\delta_\infty,\delta_n, n \leq N \}$, if $m(x,y)<\delta$, then
$|f_n(x)-f_n(y)|<\epsilon$, for any $n\in \bar\N$. In particular, for all $n\in \bar\N$
$$f_n(y) < f_n(x) + \epsilon~ \leq g(x) + \epsilon~,$$
which implies that $g(y) \leq g(x) + \epsilon$. Similarly, $g(x) \leq g(y) + \epsilon$.
\end{proof}

We will denote by $B_m(x,r)$ the open ball of center $x$ and radius $r$ for the distance $m$.

\begin{proof}[Proof of Lemma~\ref{sup metrique}]
Let $O$ be an open set of $M$.
For any $x\in O$, there is an  $\epsilon >0$ such
that $B_{d_\infty}(x,\epsilon)\subset O$.
But $B_{d_{\sup}}(x,\epsilon)\subset B_{d_\infty}(x,\epsilon) $,
so $O$ is an open set for the topology induced by $d_{\sup}$.

Let $\epsilon>0$.  Let $\bar B$ be a closed ball centred at $x$ with radius $\epsilon$ for $d_{\sup}$.
 By Fact~\ref{sup continu}, $d_{\sup}^{-1}([0,\epsilon])$ is a closed set of $M\times M$.
As $\bar B$ is the projection onto the second factor of $d_{\sup}^{-1}([0,\epsilon]) \cap (\{x\}\times M)$, it is a closed set of $M$ (by the tube lemma, as $M$ is compact, then the projection is a closed map).
\end{proof}

\subsection{Length convergence}\label{sec length}

Let $\mathfrak{L}$ be a length structure on a (connected) manifold $M$. The length structure $\mathfrak{L}$ sends  curves of a given set of curves (supposed non empty), the set of \emph{admissible curves}, to $\R$.  The pseudo-distance
$d$ induced by $\mathfrak{L}$ is defined as follows: $d(x,y)$ is the infimum of the lengths of admissible curves between $x$ and $y$.

In turn, the distance $d$ itself induces a length structure, denoted by $L_d$, and defined as follows:
the length of a curve $c:[a,b]\to M$  is defined as
\begin{equation}\label{def:met}L_d(c)=\sup_{\delta} \sum_{i=1}^{n} d(c(t_i),c(t_{i+1})) \end{equation}
where the sup is taken over all the decompositions
$$\delta = \{(t_1,\ldots,t_{n} )| t_1=a\leq t_2\leq\cdots\leq t_{n}=b\}~.$$ A curve is \emph{rectifiable} if its $L_d$-length is finite.
The length structure $L_d$ is lower-semicontinuous \cite[Proposition~2.3.4]{BBI2001}:
if a sequence of rectifiable curves $c_n:[a,b]\to M$  converges to $c$ (i.e. $c_n(t)\to c(t)$ for all $t$), then
$$L_d(c)\leq \liminf_n L_d(c_n)~. $$

In general, on the set of admissible curves (for $\mathfrak{L}$),
$$L_d\leq \mathfrak{L} $$
unless $\mathfrak{L}$ is lower-semicontinuous, as shows the following result.

\begin{proposition}[{\cite[Theorem~2.4.3]{BBI2001}}]\label{lower semic}
On the set of admissible curves, $L_d=\mathfrak{L}$ if and only if $\mathfrak{L}$ is lower-semicontinous.
\end{proposition}

In the other way, starting from a distance $d$ on $M$, it induces a length structure $L_d$, and this one induces a metric $\hat d$ on $M$. The distance $d$ is called \emph{intrinsic} if $\hat d=d$, i.e. $d(x,y)$ is the inf of the $L_d$-length of rectifiable curves between $x$ and $y$.
If $d$ comes from a length structure $\mathfrak{L}$,
then $d$ is intrinsic  \cite[Proposition~2.4.1]{BBI2001}.

\begin{lemma}\label{lem:limsup}
Let $d_n$, $n\in \N$ be intrinsic distances
on a manifold $M$, and let $\mathfrak{L}_\infty$ be a length structure inducing a pseudo-distance $d_\infty$ on $M$. Suppose  that for any $\mathfrak{L}_\infty$-admissible  curve $c$
$$L_{d_n}(c) \to \mathfrak{L}_{\infty}(c)~. $$
Then for any $x,y\in M,$
\begin{equation*}
\limsup_{n\rightarrow\infty}d_{n}(x,y)\leq d_\infty(x,y)~.
\end{equation*}
\end{lemma}
\begin{proof}
As $d_\infty$ is an infimum of length of curves, for every $\epsilon>0$ there exists a curve $c$ on $M$ connecting $x$ and $y$ such that
\begin{equation}\label{eq:curve-c-less-dist-xy+eps}
\mathfrak{L}_{\infty}(c)<d_\infty(x,y)+\epsilon~.
\end{equation}
For any $n$, $ d_{n}(x,y) \leq L_{d_n}(c)$,
then, together with the assumption of the lemma: $$\limsup_{n\rightarrow\infty} d_{n}(x,y) \leq \limsup_{n\rightarrow\infty} L_{d_n}(c)=\mathfrak{L}_{\infty}(c)~.$$

By \eqref{eq:curve-c-less-dist-xy+eps}, we get
\begin{equation*}\label{eq:limsup-less-limitdist+eps}
\limsup_{n\rightarrow\infty}d_{n}(x,y)<d_\infty(x,y)+\epsilon~.
\end{equation*}
Since $\epsilon$ is arbitrary,
the conclusion holds.
\end{proof}

\subsection{Convergence on the universal cover}\label{sec:met on univ cov}

Let us choose a point $x_o$ in the manifold $M$. We will denote by $\pi_1M$ the fundamental group
of $M$ based at $x_o$. Let $\tilde M$ be the universal cover of $M$, on which $\pi_1M$ acts by deck transformations. Let $\mathrm{\bold{p}}$ be the projection $\tilde M\to M$.

A distance $d$ on $M$ defines a length structure $\tilde{ \mathfrak{L}}$ on the set of lifts of rectifiable curves on $M$: the length of $c$ on $\tilde M$ is defined as the $L_d$-length of $\mathrm{\bold{p}}(c)$ on $M$. The length structure
 $\tilde{ \mathfrak{L}}$ defines a metric $\tilde{d}$
 on $\tilde{M}$. As $L_d$ is lower semicontinuous, $\tilde{L}$ is lower semicontinuous, and by Proposition~\ref{lower semic},
 $\tilde{L}=L_{\tilde{d}}$.

Recall the following classical result.
\begin{theorem}[{Hopf--Rinow, \cite[I.3.7]{BH1999},\cite[2.5.23]{BBI2001}}]\label{HR}
If an  intrinsic distance $d$ on  $M$ is complete and locally compact,
then for any $x,y\in M$, there exists a continuous curve joining $x$ and $y$ whose length is equal to $d(x,y)$.

Moreover, every closed bounded subset is compact.
\end{theorem}

Such a curve is called a \emph{shortest path}.
In particular, a shortest path  is rectifiable. Of course, if $d$ is given by a length structure $\mathfrak{L}$, there  is no reason why  a shortest path should be admissible for $\mathfrak{L}$.

 Let us suppose that $M$ is compact, and that $d$ is an intrinsic distance. The metric space $(\tilde M,\tilde{d})$ is complete, hence Hopf--Rinow theorem applies, see \cite[I.8.3(2)-8.4(1)]{BH1999}. Note also that closed balls are compact.

\begin{lemma}\label{lem: pro iso locale}
Let $M$ be a  compact manifold, and
$d_n,d_\infty$ distances on $M$, such that
 $d_n$ uniformly converge to $d_\infty$. For any $x\in \tilde M$ there exists an open set $U$ with $x\in U$
such that the restriction to $U$ of the projection $\mathrm{\bold{p}}:\tilde M \to M$ is an isometry for all the distances $d_n$, $n\in \bar\N$, if $n$ is sufficiently large.
\end{lemma}
\begin{proof}
Let $O$ be an open set of $\tilde M$ such that $\mathrm{\bold{p}}_{|O}:O \to \mathrm{\bold{p}}(O)$ is a homeomorphism and let $s$ be the inverse of $\mathrm{\bold{p}}_{|O}$. Then for any $r_n$ such that $B_{d_n}(x,2r_n)\subset \mathrm{\bold{p}}(O)$, $s$ is an isometry on $B_{d_n}(x,r_n)$ for $d_n$, see the proof of Proposition~I-3.25 in \cite{BH1999}.

Let $R$ be such that $B_{d_\infty}(x,2R)\subset \mathrm{\bold{p}}(O)$. Let $\epsilon >0$ with $\epsilon <R$. Then by uniform convergence of $(d_{n})_{n}$, there exists $N(\epsilon)\in\mathbb{N}$ such that for any integer $n>N(\epsilon)$,
$B_{d_n}(x,2(R-\epsilon))\subset  B_{d_\infty}(x,2R)$. 
Thus, for every $n>N(\epsilon)$, $\mathrm{\bold{p}}$ is an isometry on  $B_{d_n}(x,R-\epsilon)$.

By Lemma~\ref{sup metrique},
$B_{d_{\sup}}(x,R-\epsilon)$ is an open set, and $B_{d_{\sup}}(x,R-\epsilon)\subset B_{d_n}(x,R-\epsilon)$ for any $n\in\bar  \N$. So
we can take $U=B_{d_{\sup}}(x, R-\epsilon)$.
\end{proof}

\begin{lemma}\label{lem:upperbound cover}
Let $x,y\in \tilde M$. Under the hypothesis of Lemma~\ref{lem: pro iso locale}, for any $\epsilon>0$, for $n$ sufficiently large,
\begin{equation}\label{eq:upperbound cover-statement}
\tilde{d}_n(x,y) \leq \tilde{d}{_\infty}(x,y) + \epsilon~,
\end{equation}
in particular, $\limsup_n\tilde{d}_n(x,y)\leq \tilde{d}_{\infty}(x,y)~.$
\end{lemma}
\begin{proof}
 Let $\tilde{c}:[0,1]\to \tilde M$ be a shortest path  for the distance $\tilde{d}_\infty$ between $x$ and $y$.
By Lemma~\ref{lem: pro iso locale}, for any $p\in \tilde{M}$, there exists an
open set $U_p$ such that, up to a finite number of indices $n$,
$\mathrm{\bold{p}}$ is an isometry on $U_p$. Thus, there is an integer $k$ and  points $p_i$, $i=1,\ldots,k$ on the compact set $\tilde{c}([0,1])$ such that the $U_{p_i}$ cover $\tilde{c}([0,1])$, $p_1=x,p_{k}=y$. Let us also take some points $z_i$, $i=1,\ldots,k+1$ on $\tilde{c}([0,1])$ such that $z_1=x,z_{k+1}=y$ and $z_i\in U_{p_{i}}\cap U_{p_{i+1}}$ for $1<i<k$. Then
\begin{equation}\label{eq:upperbound cover-d<L}
\tilde{d}_\infty(x,y)=L_{\tilde{d}}(\tilde{c}) \geq \sum_{i=1}^{k} \tilde{d}(z_i,z_{i+1})~.
\end{equation}
Let $\epsilon>0$. By uniform convergence, for $n$ sufficiently large,
$$d_n(\mathrm{\bold{p}}(z_i),\mathrm{\bold{p}}(z_{i+1})) \leq d_\infty (\mathrm{\bold{p}}(z_i),\mathrm{\bold{p}}(z_{i+1}))+ \epsilon/k$$
and as $\mathrm{\bold{p}}$ is a local isometry for all the $d_n$, for $n$ sufficiently large, on any $U_i$ (Lemma~\ref{lem: pro iso locale}), we obtain
\begin{equation}\label{eq:upperbound cover-d_n-via-d}
\tilde{d}_n(z_i,z_{i+1}) \leq \tilde{d}_\infty (z_i,z_{i+1})+ \epsilon/k~.
\end{equation}
Combining the formulas~(\ref{eq:upperbound cover-d<L}) and~(\ref{eq:upperbound cover-d_n-via-d}), and applying several times the triangle inequality, we get~(\ref{eq:upperbound cover-statement}).
\end{proof}

The following lemma is Theorem~1 and Theorem~2 p.225 of \cite{AZ}.

\begin{lemma}\label{lem:con curves}
Let $M$ be a  compact manifold, and
$d_n,d_\infty$ distances on $M$, such that
 $d_n$ uniformly converge to $d_\infty$.
Let $c_n$ be curves on $M$. If there exists $L>0$ such that
for any $n$, $L_{d_n}(c_n) <L$, then, up to extract a subsequence,
the curves $c_n$ converge to a curve $c$, and
$$L_{d_\infty}(c) \leq \liminf_n L_{d_n}(c_n)~. $$
\end{lemma}

\begin{proposition}\label{prop: conv loc unif}
Let $M$ be a  compact manifold, and
$d_n,d_\infty$ distances on $M$, such that
 $d_n$ uniformly converge to $d_\infty$.
Up to extract a subsequence,
$(\tilde{d}_n)_n$ converge to $\tilde{d}_\infty$,
uniformly on compact sets.
\end{proposition}
\begin{proof}
Let $x,y\in \tilde{M}$, and $\tilde{c}_n$ be a shortest path between $x$ and $y$ for $\tilde{d}_n$, $n\in \N$. So, if $c_n=\mathrm{\bold{p}}(\tilde{c}_n)$, then
$$\tilde{d}_n(x,y)=L_{\tilde{d}_n }(\tilde{c}_n)=
L_{d_n}(c_n)~.$$
By Lemma~\ref{lem:upperbound cover}, $L_{d_n}(c_n)$ are uniformly bounded from above for $n$ sufficiently large. By Lemma~\ref{lem:con curves}, up to extract a subsequence, there is a curve $c$ between $x$ and $y$ with
$$\tilde{d}_\infty(x,y)\leq L_{d_\infty}(c) \leq \liminf_n L_{d_n}(c_n)=\liminf_n  \tilde{d}_n(x,y)~. $$
This and Lemma~\ref{lem:upperbound cover}
give the pointwise convergence. The result follows from Lemma~\ref{lem alex trick}:
 by the definition of the length of the curve,  for any $n\in \bar \N$,
$ \tilde{d}_n \leq \tilde{d}_{\sup}~.$

\end{proof}

\begin{lemma}\label{lem boules bornees}
Let $M$ be a  compact manifold, and
$d_n,d_\infty$ distances on $M$, such that
 $d_n$ uniformly converge to $d_\infty$. Then for any $r>0$ and any $x\in \tilde{M}$,
 $\cup_{n\in \bar{\N}}B_{\tilde{d}_n}(x,r)$ is bounded in $\tilde{M}$ for $\tilde{d}_\infty$.
\end{lemma}
\begin{proof}
Suppose the contrary, that is, for any $k\in \N$, there exists $y_k\in B_{\tilde{d}_{n_k}}(x,r)$ such that $\tilde{d}_\infty(x,y_{n_k}) >k$. Let $c_{n_k}$ be a shortest
path between $x$ and $y_{n_k}$ for $\tilde{d}_{n_k}$. Then
we have
$$L_{\tilde{d}_{n_k}}(\bold{p}(c_{n_k}))=L_{\tilde{d}_{n_k}}(c_{n_k})=\tilde{d}_{n_k}(x,y_{n_k}) <r $$
so  by Lemma~\ref{lem:con curves} there is a curve $c$ between $x$ and $y$, where $y_{n_k}\to y$,  up to extract a subsequence,
and
$$\tilde{d}_\infty(x,y)\leq L_{d_\infty}(\bold{p}(c)) \leq \liminf_{n_k} L_{d_{n_k}}(\bold{p}(c_{n_k}) \liminf_{n_k}  \tilde{d}_{n_k}(x,y_{n_k})<r~, $$
that is a contradiction.
\end{proof}

\begin{lemma}\label{lem:ratcliffe}
Let $d$ be a distance on the compact manifold $M$. There exists $l>0$ such that for any $\gamma\in (\pi_1M\setminus \{0\})$, for any $x \in \tilde{M}$, $B_{\tilde{d}}(x,l)$ and $B_{\tilde{d}}(\gamma.x,l)$ are disjoint.
\end{lemma}
\begin{proof}
The proof is formally the same as the one of Lemma~1 p.~237 in \cite{Rat06}.
\end{proof}

\begin{lemma}\label{lem: dans dom find}
Let $M$ be a  compact manifold, and
$d_n,d_\infty$ distances on $M$, such that
 $d_n$ uniformly converge to $d_\infty$.
Let $F$ be  a fundamental domain in $\tilde{M}$ for the action of $\pi_1M$, with compact closure.
There exists $G>0$ and $N>0$ such that, for every element $\gamma\in(\pi_{1}M\setminus{\{0\}})$,  for any $x\in F$,  for every $n>N$,
$$ \tilde{d}_{n}(x,\gamma. x)\geq G~.
$$
\end{lemma}
\begin{proof}
Let $l>0$ given by Lemma~\ref{lem:ratcliffe} applied to $\tilde{d}_\infty$, and let $\epsilon >0$ such that $l-\epsilon >0$.
Let $x_1\in F$ and let $R$ be such that the closure of $F$ is contained
in $B_{\tilde{d}_{\operatorname{sup}}}(x_1,R)$. So for any $n\in \bar \N$, the closure of $F$ is contained
in $B_{\tilde{d}_{n}}(x_1,R)$.
By Lemma~\ref{lem boules bornees},  $\cup \{B_{\tilde{d}_n}(x_1,R+l) | n\in \bar{\N} \}$ is contained in a compact set $K$.
It follows that $\cup \{B_{\tilde{d}_n}(x,l) | x\in F, n\in \bar{\N} \}$ is contained in $K$.

By Proposition~\ref{prop: conv loc unif}, there exists $N$ such that for any $n>N$ and $x,y\in K$,
$\tilde{d}_\infty(x,y) < \tilde{d}_n(x,y)+\epsilon $.

In particular, if $n>N$,  for any $x\in F$, $B_{\tilde{d}_n}(x,l-\epsilon) \subset B_{\tilde{d}_\infty}(x,l)$. As $\pi_1M$ acts by isometries on $(\tilde{M},\tilde{d}_n)$ and  $(\tilde{M},\tilde{d}_\infty)$,
we also have, for any $\gamma\in \pi_1M$,  $B_{\tilde{d}_n}(\gamma.x,l-\epsilon) \subset B_{\tilde{d}_\infty}(\gamma.x,l)$. Lemma~\ref{lem:ratcliffe} then implies that $B_{\tilde{d}_n}(\gamma.x,l-\epsilon)$ and $B_{\tilde{d}_n}(x,l-\epsilon)$ are disjoint if $\gamma\not= 0$,
hence $\tilde{d}_n(x,\gamma.x)> l-\epsilon >0$.
\end{proof}

\begin{corollary}\label{lem:min-length-gamma>K-for-all-n}
Let $M$ be a  compact manifold, and
$d_n,d_\infty$ distances on $M$, such that
 $d_n$ uniformly converge to $d_\infty$.
There exists $G>0$ and $N>0$ such that, for every element $\gamma\in(\pi_{1}M\setminus{\{0\}})$,  for any $x\in\tilde{M}$, for any $n>N$,
$$ \tilde{d}_{n}(x,\gamma. x)\geq G~.
$$
\end{corollary}
\begin{proof}
Let $F$ be  a fundamental domain in $\tilde{M}$ for the action of $\pi_1M$.
Let $x\in \tilde{M}$. Then there exists a $\mu\in \pi_1M$ such that $\mu. x\in F$.
Now let $\gamma\in(\pi_{1}M\setminus{\{0\}})$.
We have, as $\pi_1M$ acts by isometries on $(\tilde{M},\tilde{d}_n)$:
 $$ \tilde{d}_{n}( x,\gamma.x)=\tilde{d}_{n}(\mu. x,\mu. \gamma.x)
 =\tilde{d}_n(\mu.x, (\mu\gamma\mu^{-1}).(\mu . x))~,$$ and by
Lemma~\ref{lem: dans dom find}, this last quantity is $\geq G$.
\end{proof}

\subsection{Approximation by polyhedral metrics} \label{paragraph:construction-of-polyhedral-metrics}

Recall (see e.g. \cite{BH1999,BBI2001,AKP}) that for any
 triple of points $(x,y,z)$ in an intrinsic metric space $(M,m)$, a \emph{comparison triangle} is a triangle on the Euclidean plane with vertices $(x',y',z')$ such that $m(x,y)=d_{\R^{2}}(x',y')$, $m(y,z)=d_{\R^{2}}(y',z')$, and $m(x,z)=d_{\R^{2}}(x',z')$.
 Using comparison triangles, it is possible to  define a notion of \emph{upper angles} between two geodesic paths in $(M,m)$
starting from the same point
(see I.1.12 in \cite{BH1999}).
Following Proposition~II.1.7 in  \cite{BH1999} (parts ($1$) and ($4$)), we give a definition of CAT($0$) space in the form which is convenient for us.
\begin{definition}\label{def-CAT0-via-angle}
A complete intrinsic metric space $(M,m)$ is CAT(0) if the upper angle between any couple of sides of every geodesic triangle  with distinct vertices is no greater than the angle between the
corresponding sides of its comparison triangle in $\R^2$.
\end{definition}

The CAT($0$) condition implies that the shortest path between two points of $M$ is unique \cite[II 1.4]{BH1999}.

\begin{definition}\label{def: cba}
A complete intrinsic metric space $(M,m)$
is  of \emph{non-positive curvature} (in the Alexandrov sense), if for any $x$ there is $r$ such that $B_m(x,r)$ endowed with the induced (intrinsic) distance is CAT(0).
\end{definition}

Note that many compact metric spaces cannot be CAT($0$) as this condition implies that the space must be simply connected  \cite[II-1.5]{BH1999}.
But by the Cartan--Hadamard theorem, if $M$ is a manifold and $(M,m)$ is of non-positive curvature, then $(\tilde M,\tilde{m})$ is CAT($0$) \cite[II.4.1]{BH1999}.

We will now consider that $M=S$ is a compact surface.
Recall from the introduction that a polyhedral metric of non-positive curvature on $S$ is
 a flat metric on $S$ with conical singularities of negative curvature.
 They admit a geodesic triangulation whose vertices are exactly the singular points \cite{troyanov}, so equivalently they can be defined as a gluing of flat triangles along isometric edges, such that the sum of the angles of
 triangles around each vertex is $>2\pi$.
It follows from \cite[Lemma~II.5.6]{BH1999} that polyhedral metrics of non-positive curvature
are metrics of non-positive curvature in the sense of Definition~\ref{def: cba}.

Let us remind the notion of  \emph{bounded integral curvature} (or just of \emph{bounded curvature} in terms of \cite[Chapter~I, p.~6]{AZ}).
\begin{definition}\label{def-integral-bounded-curvature}
An intrisic distance $m$ on a surface $S$ is said to be of \emph{bounded integral curvature} (in short, BIC) if $(S,m)$ verifies the following property:

\begin{itemize}
  \item \emph{
For every $x\in S$ and every neighbourhood $N_x$ of $x$ homeomorphic to the open disc,  for any  finite system $\{T\}$ of pairwise nonoverlapping \emph{simple} triangles $T$ belonging to $N_x$, the sum of the \emph{excesses}
\begin{equation*}\label{eq:def-excess-triangle}
\delta(T)=\bar{\alpha}_{T}+\bar{\beta}_{T}+\bar{\gamma}_{T}-\pi
\end{equation*}
of the triangles $T\in\{T\}$ with upper angles $(\bar{\alpha}_{T},\bar{\beta}_{T},\bar{\gamma}_{T})$, is bounded from above by a number $C$ depending only on the neighbourhood $N_x$, i.e.
}
\begin{equation*}\label{eq:def-integral-curvature}
\sum_{T\in\{T\}}\delta(T)\leq C~.
\end{equation*}
\end{itemize}
\end{definition}

A simple triangle is a triangle bounding an open set homeomorphic to a disc, and which is \emph{convex relative to the boundary}, i.e. no two points of the boundary of the triangle can be joined by  a curve outside the triangle which is shorter than a suitable part of the boundary joining the points, see \cite{AZ} for more details.

\begin{lemma}
A metric of non-positive curvature on a compact surface is BIC.
\end{lemma}
\begin{proof}
By definition, each point $x$ of a metric of non-positive curvature on a compact surface space has a CAT($0$) neighbourhood $N_x$. Therefore, by Definition~\ref{def-CAT0-via-angle}, for any geodesic triangle $T\subset N_x$, the angles of the comparison triangle $\overline{T}\subset\R^{2}$ are not less than the corresponding upper angles of $T$. As the sum of the angles of a Euclidean triangle is $\pi$, the excess of the triangle $T$ (see Definition~\ref{def-integral-bounded-curvature}) is
\begin{equation*}\label{eq:excess-CAT(0)-triangle}
\delta(T)\leq\delta(\overline{T})=0~.
\end{equation*}
Hence, for any finite system $\{T\}$ of pairwise nonoverlapping simple triangles $T$ in $N_x$, the sum of the excesses is trivially non-positive.
\end{proof}

We want to find a sequence $d_n$ of polyhedral metrics of non-positive curvature on a compact surface $S$ converging to a given distance $d$ of non-positive curvature on
 $S$.

 The main tool is Theorem~10 in \cite[Chapter~III, p.~84]{AZ}. This result adapted to our case is  formulated as follows.
We call a \emph{flat metric} on a surface a distance such that each point has a neighbourhood isometric to a Euclidean cone. The cone has  arbitrary angle at the vertex. If the angle is $2\pi$, the point is regular, and singular otherwise. If the angles at the singular points are $>2\pi$,  a flat metric is a  polyhedral metric of non-positive curvature.

\begin{theorem}\label{thm:az-polyhedral-approximation}
Given a compact BIC surface, there is a sequence of flat metrics converging uniformly to it.
\end{theorem}

We will also need the following technical result, which corresponds to Theorem~11 in \cite[Chapter~II, p.~47]{AZ}. Here we write it down in a convenient form.
\begin{lemma}\label{thm:az-point-on-geodesic}
Let $p$ be a point on a BIC surface such that there is at least one shortest arc   containing $p$ in its interior. Then for any decomposition of a neighbourhood of
 $p$ into sectors convex relative to the boundary formed by geodesic rays issued from $p$ such that the upper angles between the sides of these sectors exist and do not exceed $\pi$,
the total sum of those angles
is not less than $2\pi$.
\end{lemma}

\begin{corollary}\label{thm:approx}
Let $d$ be a distance of non-positive curvature on a compact surface $S$.
Then there exists a sequence of polyhedral metrics of non-positive curvature  $(d_{n})_{n\in\mathbb{N}}$ on $S$ which uniformly converges to $d$.
\end{corollary}
\begin{proof}

Applying Theorem~\ref{thm:az-polyhedral-approximation}, we obtain a sequence $(d_{n})_{n\in\mathbb{N}}$, of flat metrics converging uniformly to $d$. We have to check that the total angles around the conical singularities of $d_{n}$, ${n\in\mathbb{N}}$, are not less than $2\pi$.

In the proof of Theorem~\ref{thm:az-polyhedral-approximation} the distances $d_n$, ${n\in\mathbb{N}}$, are constructed as follows:
\begin{itemize}
  \item[$1.$] construct a geodesic triangulation $\tau_{n}$ of  $(S,d)$;
  \item[$2.$] replace the interiors of the triangles of $\tau_{n}$ by the interiors of the Euclidean comparison triangles.
\end{itemize}
Note that point $1.$ is far from being trivial. Then, one has to prove that
the finer the triangulation is, the closer $d_n$ is from $d$ (for the uniform distance between metric spaces).

Remark that, by construction (see the proof of Theorem~10 in \cite[Chapter~III, p.~85, lines~3 and~4]{AZ}), every vertex of $\tau_{n}$  lies in the interior of some geodesic in $(S,d)$ (this also follows because $(S,d)$ is of non-positive curvature \cite[II.5.12]{BH1999}). Applying Lemma~\ref{thm:az-point-on-geodesic}, we immediately get that the sum of the sector angles at any vertex $V$ of a triangulation $\tau_{n}$ in $(S,d)$, is not less than $2\pi$.  By Definition~\ref{def-CAT0-via-angle}, the angles of the comparison triangles in $\R^2$ are not less than the corresponding sector angles at $V$ of the triangulation $\tau_{n}$ in $(S,d)$. Hence  the total angle around every singular point of the polyhedral metric $d_{n}$  is not less than $2\pi$.
\end{proof}


\section{Convergence of induced distances}\label{sec 3}

The aim of this section is to prove Proposition~\ref{prop:main 3}.

\subsection{Graphs on the hyperboloid}

Let $K$ be a spacelike convex set, and suppose that its boundary
$G_u$ is the graph of a positive function $u$ on $\H^2$, i.e.

\begin{equation}\label{eq Gu}G_u=\{u(x)x | x\in \H^2 \}~. \end{equation}
Note that $G_u\subset I^+(0)$.

\begin{definition}
A function $u:\H^2\to \R$ is \emph{H-convex} if $G_u$ defined by
\eqref{eq Gu} is a spacelike convex surface.
\end{definition}

The function $u:\H^2\to \R$ can be written as
$$u(x)=\mbox{inf}\{\lambda \geq 0 | \lambda x \in K \} $$
so it is the restriction to $\H^2$ of the function $U:I^+(0)\to \R$ defined by
$$U(x)=\mbox{inf}\{\lambda \geq 0 | \lambda x \in K \}~. $$
Let $\mu \geq 0$. As $U(x)x\in \partial K$, then
$\frac{U(x)}{\mu}\mu x \in \partial K$. Also,
$U(\mu x)\mu x \in \partial K$.
As for any $y\in I^+(0)$, there exists exactly one $\lambda >0$ such that $\lambda y \in \partial K$, it follows that $U(\mu x)=\frac{U(x)}{\mu}$, i.e.
$U$ is $(-1)$-homogeneous.

Also, if $x\in K$ and $\lambda \geq 1$, then $\lambda x \in K$, so
 $x\in \partial K$ if and only if $U(x)=1$. If $x\in K$, there exists $\lambda \leq 1$ such that $\lambda x \in \partial K$. In particular, if $x\in K$ and $\lambda \leq 1$,
then $U(\lambda x)=1$ i.e. $U(x)=\lambda\leq 1$, so
\begin{equation*}\label{eq K U}K=\{x | U(x) \leq 1 \}~. \end{equation*}

\begin{lemma}\label{lem: U}
Let $u$ be a H-convex function, and let $U$ be its $(-1)$-homogeneous extension on $I^+(0)$.  The function $-\frac{1}{U}$ is convex.
\end{lemma}
\begin{proof}
By definition of $U$,
\begin{equation}\label{eq def encore K}\frac{1}{U}(z)=\mbox{max}\{t \geq 0 | z\in t K\}~. \end{equation}
Let $x,y\in I^+(0)$. Hence $x \in \frac{1}{U}(x)K$ and $y \in \frac{1}{U}(y)K$,
so, as $K$ is convex, for any $\lambda,\mu$, $\lambda K + \mu K=(\lambda +\mu) K$ \cite[Remark 1.1.1]{schneider}, hence, for $0\leq a \leq 1$,
$$ax + (1-a)y \in \left( a\frac{1}{U_n}(x)+(1-a)\frac{1}{U_n}(y)\right) K $$
and by \eqref{eq def encore K},
$$a\frac{1}{U}(x)+(1-a)\frac{1}{U}(y)\leq \frac{1}{U}(ax+(1-a)y)$$
hence $\frac{1}{U}$ is concave.
\end{proof}

 We will use the strong fact that the function $-\frac{1}{U}$ is convex.

\begin{lemma}\label{lem conv un}
Let $(u_n)_n$ be a sequence of H-convex functions, such that there exists
 $\beta>\alpha>0$ with $\alpha<u_n<\beta$. Up to extract a subsequence, $(u_n)_n$ converges to a H-function $u$, uniformly on compact sets.
\end{lemma}
\begin{proof}
Let $X\in W \subset I^+(0)$, where $W$ is a compact set that does not touch $\partial I^+(0)$. By $(-1)$-homogeneity of the $U_n$,
$$U_n(X)=(-\langle X,X\rangle_-)^{-1/2}u_n\left( \frac{X}{(-\langle X,X\rangle_-)^{1/2}}\right)~, $$
so as for any $x\in \H^2$, the sequence $(u_n(x))_n$ is bounded, and as $W\cap\partial I^+(0)=\emptyset$,
the sequence $(-\frac{1}{U_n}(X))_n$ is bounded, and
by standard property of convex functions,
\cite[10.9]{Roc97}, there is a function $U$ such that, up to extract a subsequence,
$(-\frac{1}{U_n})_n$ converges uniformly on each compact subsets of the interior of $W$ to a convex function $-\frac{1}{U}$. We do this for any compact neighborhood of points of $I^+(0)$ that do not touch $\partial I^+(0)$. Hence we obtain a convex function $-\frac{1}{U}:I^+(0)\rightarrow\R$. The set
$$K=\{x | -\frac{1}{U(x)} \leq -1 \}$$
is convex, and, as  $U$ is clearly $(-1)$-homogeneous,
$$ \partial K=\{x | -\frac{1}{U(x)} = -1 \}=\{x| U(x)=1 \}$$
is a convex surface, and it is the graph of the restriction $u$ of $U$ on $\H^2$. Obviously, $\alpha\leq u\leq\beta$.

Support planes of $\partial K$ are limits of support planes of
$\partial K_n$, hence spacelike or lightlike.
Let us suppose that $\partial K$ has a lightlike support plane $P$.
Then $\partial K$ is on one side of $P$, but meets all the hyperboloids centred at zero. This contradicts the fact that $\partial K$ is a graph above such an hyperboloid (because $u\geq\alpha$).
\end{proof}

\begin{lemma}\label{lem:equilip}
The functions $u_n$ and $u$ of Lemma~\ref{lem conv un} are equi-Lipschitz on any compact set of $\H^2$.
\end{lemma}
\begin{proof}
If $U_n$ and $U$ are the $(-1)$ homogeneous extensions of $u_n$ and $u$ respectively, we know by Lemma~\ref{lem: U}
that $-1/U_n$ and $-1/U$ are convex.

As $H_n=-1/U_n$ is convex
on $I^+(0)$, then, for a compact set $C\subset \H^2$, there exists a $\epsilon >0$ such that for all $x,y\in C$,  \cite[10.4]{Roc97}
(here $\|\cdot\|$ is the Euclidean norm on $\R^3$)

$$|H_n(x)-H_n(y)|\leq \frac{\operatorname{max}_C H_n -\operatorname{min}_C H_n}{\epsilon} \|x-y\| $$
that leads to

$$|u_n(x)-u_n(y)| \leq \frac{\operatorname{max} u_n}{\epsilon}\left(\frac{\operatorname{max} u_n}{\operatorname{min} u_n}-1\right)\| x-y\| $$
 and as $\alpha<u_n<\beta$,  the $(u_n)_n$ are equi-Lipschitz on $C\cap\H^2$, for the distance on $\H^2$ induced by the ambient Euclidean one. But all the norms coming from Riemannian structures are locally equivalent on $\H^2$.
\end{proof}

\subsection{Convergence of the length structures}

This part is a straightforward adaptation of classical results from the Euclidean setting \cite{alex}.
Let $u$ be a H-convex function.
Let  $c:[a,b]\rightarrow \H^2$ be a Lipschitz curve.
Let $v=(u\circ c)c$.
When the derivative exists,
 $$v'=(u\circ c)'c+u(c)c'$$
 but $c'(s)$ belongs to $T_{c(s)}\H^2$, which is orthogonal to
 $c(s)$, and $\langle c(s),c(s)\rangle_-=-1$, and
  $$\langle c'(s),c'(s)\rangle_-^{1/2}=\|c'\|_{\H^2}$$ is the norm induced by
  the hyperbolic metric. So
 \begin{equation*}
\langle v',v'\rangle_-=u^2(c)\|c'\|_{\H^2}^2 -((u\circ c)')^2
\end{equation*}
and we define the following length
\begin{equation}\label{eq:length}
\mathfrak{L}_{u}(c)=\int_a^b \left(u^2(c)\|c'\|_{\H^2}^2 - ((u\circ c)')^2 \right)^{1/2}~.
\end{equation}

Note that as $\H^2$ is a smooth hypersurface,
$\mathfrak{L}_1=L_{d_{\H^2}}$ on the set of Lipschitz curves, see e.g. \cite{bur}.
The following is immediate.

\begin{lemma}\label{lem: comp hyp met1}
Let $u$ be a H-convex function and $\beta>0$ with $u < \beta$.
For any Lipschitz curve $c$ on $\H^2$,
$$ \mathfrak{L}_{u}(c)\leq \beta L_{d_{\H^2}}(c)~.$$
\end{lemma}

Let $u_n$ be H-convex functions converging to a H-convex function $u$.
As the $u_n$ are Lipschitz on $\H^2$, by Rademacher theorem, they are
differentiable almost everywhere (for the Borel measure given by the hyperbolic metric). As there is a countable number of
$u_n$,  there exists a set $\mathcal{D} \subset \H^2$
of zero measure, such that
 the $u_n$ and $u$ are
differentiable on $\H^2\setminus \mathcal{D}$.

Let $c:I\to \H^2$ be a Lipschitz curve. The subset $c^{-1}(\mathcal{D})$ may be a set of non-zero measure in $I$. However, $u\circ c$ is a Lipschitz function on $I$, hence it is derivable almost everywhere on $I$, and moreover, all the $u_n\circ c$
are simultaneously derivable almost everywhere.
To illustrate, let us consider the example of the union of two halfplanes meeting along a line.
It is the graph of a function $f$, which is differentiable everywhere except on the
projection $c$ of the edge onto the plane. But the restriction of $f$ to $c$ is derivable everywhere on $c$.

\begin{lemma} \label{lem: con der} Let $u_n$ be H-convex functions converging to a H-convex function $u$.
For almost all $t\in I$, $u_n$ and $u$ are derivable and at such a point $t$, up to extract a subsequence,
$u_n(c(t))' \to u(c(t))'$.
\end{lemma}
\begin{proof}
Let us denote by $X$ the unit timelike vector $c(t)$ and
by $Y$ the unit spacelike vector $c'(t)$. As $c$ is a curve on $\H^2$, $\langle X,Y \rangle_-=0$.
The tangent vector of the curve $(u_n\circ c)c$ is
$$V_n=(u_n(c(t)))'X+u_n(c(t))Y $$
and in the plane $P$ spanned by $X$ and $Y$, the vector
$$N_n=u_n(c(t))X+(u_n(c(t)))'Y $$
is orthogonal to $V_n$ for $\langle\cdot,\cdot\rangle_-$.
It follows from Lemma~\ref{lem:equilip} and the fact that the $u_n$ are uniformly bounded,  that
the Euclidean norms of $N_n$ are uniformly bounded.
Hence, up to extract a subsequence, $(N_n)_n$ converges to a vector $N$. Note that
$N$ is non-zero, otherwise $\langle N_n,X\rangle_-=-u_n(c(t))$ would converge to $0$, that is impossible since $0<\alpha< u_n <\beta$.

Let $A_n$ be the intersection of the convex set $K_n$ defined by $u_n$  and
the plane $P$. The set $A_n$ is a convex set, and  $V_n$ is a  tangent vector, hence, for any $y \in \H^2\cap P$,
$$\langle N_n, u_n(c(t))X - u_n(y)y \rangle_- \geq 0 $$
and passing to the limit,
$$\langle N, u(c(t))X - u(y)y \rangle_- \geq 0 $$
that says that $N$ is a normal vector of $A$ (the intersection of $K$ with $P$),
in particular,  $N$ is orthogonal  to
$$V=(u(c(t)))'X+u(c(t))Y~. $$
It follows that  there exists $\lambda$ such that
$$\lambda N=u(c(t))X+(u(c(t)))'Y$$
but as $\langle N_n,X\rangle_-$ converges to
$\langle N,X\rangle_-$ and $u_n(c(t))$ converges to
$u(c(t))$, then $\lambda=1$ and $(u_n(c(t)))'$ must converge to $(u(c(t)))'$.
\end{proof}

The preceding result and the Dominated convergence theorem give the following.

\begin{proposition}\label{prop:conv leng}
For any Lipschitz curve $c:[a,b]\to \R$ on $\H^2$, if $u_n\to u$, then, up to extract a subsequence, $\mathfrak{L}_{u_n}(c)\to \mathfrak{L}_u(c)$.
\end{proposition}

\subsection{Estimates for the induced distances}

In the preceding section, we defined a length structure
$\mathfrak{L}_u$. Let $d_u$ be the distance
defined by this length structure.
To be more precise, for the moment $d_u$ is only a pseudo-distance. As we noticed in Section~\ref{sec length}, the pseudo-distance $d_u$ is intrinsic.
We have a first bound, immediate from Lemma~\ref{lem: comp hyp met1}.

\begin{lemma}\label{lem: comp hyp met12}
If $u\leq \beta$, then $d_{u} \leq \beta d_{\H^ 2}~.$
\end{lemma}

\begin{lemma}\label{eq:principale}
Let $u$ be a  H-convex function such that $d_u$ is a complete distance with Lipschitz shortest paths. Let $\alpha>0$ with $u > \alpha$.  Then

\begin{equation*}d_{\H^2}(x,y)\leq \frac{1}{\alpha} d_{u}(x,y) +
\frac{1}{\alpha^2}\int_0^{d_{u}(x,y)}\sqrt{\langle \nu(t),\nu'(t)\rangle_{-}^2}\operatorname{d}t~,\end{equation*}
where $\nu$ is an arc-length parametrized shortest path on the graph of $u$ between $u(x)x$ and $u(y)y$.
\end{lemma}

\begin{proof}
Let $\pr(y)=\frac{y}{\sqrt{-\langle y,y\rangle_-}}$ be the radial projection from the future cone of the origin of Minkowski space onto $\H^2$.
By definition of the length, we have that
\begin{equation*}
d_{\H^2}(x,y)  \leq L_{d_{\H^2}}(\pr\circ \nu)
\end{equation*}
where
$$L_{d_{\H^2}}(\pr\circ \nu) :=  \int_0^{d_{u}(x,y)} \|(\pr\circ \nu)'(t)\|_{\H^2}\operatorname{d}t= \int_0^{d_{u}(x,y)} \sqrt{\langle (\pr\circ \nu)'(t),(\pr\circ \nu)'(t)\rangle_{-}} \operatorname{d}t~.$$
A straightforward computation gives
$$L_{d_{\H^2}}(\pr\circ \nu)=\int_0^{d_{u}(x,y)} \sqrt{\frac{\langle \nu'(t),\nu'(t)\rangle_{-}}{-\langle \nu(t),\nu(t)\rangle_{-}}+\frac{\langle \nu(t),\nu'(t)\rangle_{-}^2}{\langle \nu(t),\nu(t)\rangle_{-}^2}}\operatorname{d}t~,$$
so, as  $\sqrt{a^2+b^2}\leq a+b$ for $a>0$ and $b>0$,
\begin{equation*}
d_{\H^2}(x,y) \leq\int_{0}^{d_{u}(x,y)}\sqrt{\frac{\langle \nu'(t),\nu'(t)\rangle_{-}}{-\langle \nu(t),\nu(t)\rangle_{-}}}\operatorname{d}t+\int_{0}^{d_{u}(x,y)}\sqrt{\frac{\langle \nu(t),\nu'(t)\rangle_{-}^2}{\langle \nu(t),\nu(t)\rangle_{-}^2}}\operatorname{d}t~.
\end{equation*}

For every $y\in\H^2$ we have $\langle y,y\rangle_-=-1$, also, for each $t$ there is $x\in\H^2$ such that $\nu(t)=u(x)x$, and $u>\alpha$ by assumption, therefore $-\langle \nu(t),\nu(t)\rangle_-\geq \alpha^2$ for all $t$, and so
$$
d_{\H^2}(x,y) \leq
\frac{1}{\alpha}
\int_{0}^{d_{u}(x,y)}\sqrt{\langle \nu'(t),\nu'(t)\rangle_{-}}\operatorname{d}t+\frac{1}{\alpha^2}\int_{0}^{d_{u}(x,y)}\sqrt{\langle \nu(t),\nu'(t)\rangle_{-}^2}\operatorname{d}t~.
$$
As $\int_{0}^{d_{u}(x,y)}\sqrt{\langle \nu'(t),\nu'(t)\rangle_{-}}\operatorname{d}t=d_{u}(x,y)$, we obtain the result.

\end{proof}

\begin{lemma}\label{lem tech}
Let $u$ be a H-convex function, with $\alpha,\beta>0$ such that $\alpha <u<\beta$,  and let $\nu$
be a path on $G_u$, and $t$ such that  $\nu'(t)$ exists. Suppose that $\nu$ is parametrized such that $\nu'(t)$ has unit norm.
Then $$\langle \nu(t),\nu'(t)\rangle_{-}^2\leq  \beta^2-\alpha^2~. $$
\end{lemma}

  \begin{proof}

Let $C$ be the closure of the cone of all the
lines 
containing the point $\nu(t)$ which meet $\beta\H^2$.
Let $C^0$ be the closure of the complement of $C$.

Note that $\nu'(t)$  lies in a spacelike plane $P$ which is a support plane of $G_{u}$. As this surface is in the past of $\beta \H^2$, the plane $P$ is in the past of the plane tangent to
$\beta\H^2$ which is parallel to $P$.
In particular, the line from $\nu(t)$ and directed by $\nu'(t)$ never crosses $\beta\H^2$, i.e. is in $C^0$.

Let $I(\nu(t))$ be the isotropic cone at $\nu(t)$. Note that $I(\nu(t))\subset C$.

Let $Q$ be the Minkowski plane containing $\nu(t)$ and
$\nu'(t)$, and let $dS$ be the connected component of the set of unit spacelike vectors in $Q$ centered at $\nu(t)$ that contains $\nu'(t)$, see Figure~2. Let $\varpi$ be a vector of $dS$ such that $\langle \nu(t),\varpi\rangle_- = 0$. The function $f=\langle \nu(t),\cdot\rangle_{-}^{2}$ defined on $dS$ is non-negative and increases monotonically to $+\infty$ when the argument moves along $dS$ from the point $\varpi$ and approaches either the future or the past component of the isotropic cone $I(\nu(t))$.
So the restriction of $f$ to $C^0\cap dS$ attains its maximal value at a vector $v$ of $\partial C^0\cap Q$.
In particular, $\langle \nu(t),\nu'(t)\rangle_{-}^{2}$ is bounded from above by the positive quantity $\langle \nu(t),v\rangle_{-}^{2}$.

By definition,
the line from the point $\nu(t)$ directed by the vector $v$ is tangent to $\beta\H^2$. Such a vector $v$ is defined by the fact that there exists $s\in \R$ satisfying
\[
\left \{
\begin{array}{c @{=} c}
    \langle \nu(t)+sv,\nu(t)+sv\rangle_- & -\beta^2 \\
\langle v,\nu(t)+sv\rangle_-     & 0 \\
\end{array}
\right.
\]
that gives
$\langle \nu(t),v\rangle_-^2=\langle \nu(t),\nu(t)\rangle_-+\beta^2 $.
As $\nu(t)$ is in the future of $\alpha\H^2$,
$$\langle \nu(t),\nu(t)\rangle_-\leq -\alpha^2~.$$ At the end,
$$\langle \nu(t),\nu'(t)\rangle_{-}^2\leq \langle \nu(t),v\rangle_-^2 \leq  \beta^2-\alpha^2~. $$

 \end{proof}

\begin{figure}[h]\begin{center}
\includegraphics[scale=0.20]{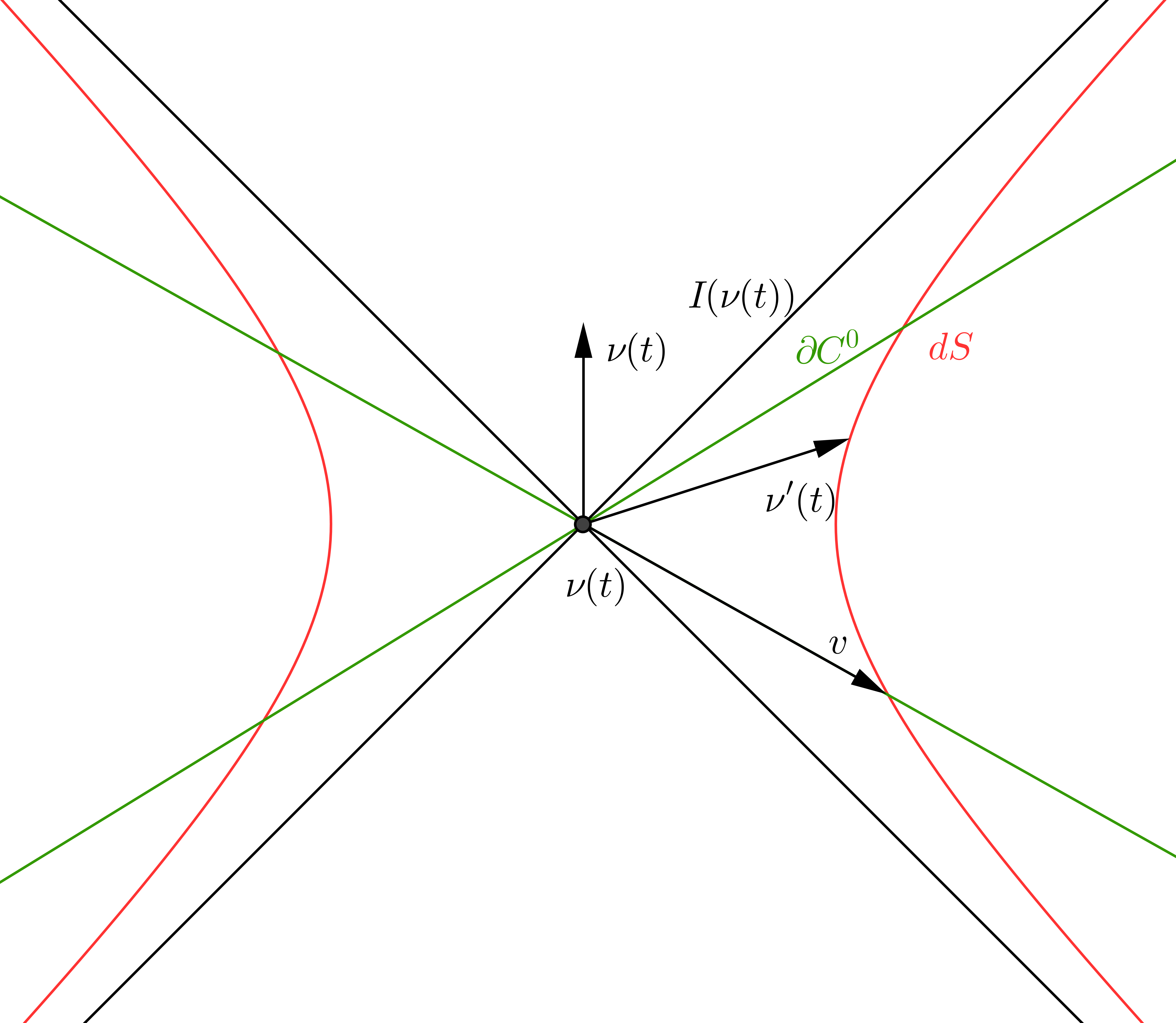}
\end{center}\label{fig3}\caption{To the proof of Lemma~\ref{lem tech}}\end{figure}

\begin{corollary}
\label{lem:upper hyp dist}
Let $u$ be a  H-convex function such that $d_u$ is a complete distance with Lipschitz shortest paths.  Then
$d_u$ is bi-Lipschitz equivalent to $d_{\H^2}$. More precisely, if $\alpha,\beta>0$ are such that
$\alpha < u <\beta$, then
\begin{equation}\label{eq up bound dh}\left(\frac{1}{\alpha}  +
\frac{\sqrt{\beta^2-\alpha^2}}{\alpha^2}\right)^{-1}d_{\H^2}\leq d_{u} \leq \beta d_{\H^2}~. \end{equation}
\end{corollary}

\begin{proposition}\label{prop:con mink}
Let $(u_n)_n$ be a sequence of  H-convex functions  such that:
\begin{itemize}
\item $d_{u_n}$ is a complete distance with Lipschitz shortest paths;
\item $\mathfrak{L}_{u_n}=L_{d_{u_n}}$ on the set of Lipschitz curves;
\item there are $\beta>\alpha>0$ such that $\alpha<u_n<\beta$.
\end{itemize}
Then, up to extract a subsequence,
$(u_n)$ converges to a H-convex function $u$, and $(d_{u_n})_n$ converges to $d_u$, uniformly on compact sets.
\end{proposition}

Note that by \eqref{eq up bound dh},
$d_u$ is a complete distance on $\H^2$.

\begin{proof}
The existence of $u$ follows from Lemma~\ref{lem conv un}. Let $\epsilon >0$, $x,y \in \H^2$ and  $\delta=(t_1,\ldots,t_m)$  a decomposition of $[0,1]$. Let $c_n$ be a shortest path for $d_{u_n}$ between $x$ and $y$.
By Corollary~\ref{lem:upper hyp dist}, there are positive constants $k_1,k_2$ such that
$$L_{k_1d_{\mathbb{H}^2}}(c_n) \leq L_{d_{u_n}}(c_{n})=d_{u_n}(x,y) \leq k_2d_{\mathbb{H}^2}(x,y)$$ so, the $d_{\mathbb{H}^2}$-length of the $c_n$ is bounded from above independently of $n$. Let us also denote by  $c_n$ a reparametrization of $c_n$ defined on $[0,1]$ proportionally to the hyperbolic arc-length. As the endpoints of the curve are fixed, all the curves $c_n$ are contained in a closed ball for the metric $d_{\mathbb{H}^2}$. 
So we can apply Arzela--Ascoli theorem (\cite[1.4.10]{papadop}, \cite[2.5.14]{BBI2001}): up to extract a subsequence,
$(c_n)_n$ uniformly converges to a curve $c_\infty$.  Note that the $c_n$ are equi-Lipschitz, hence $c_\infty$ is Lipschitz.
So,  there is $N_1$ such that if $n\geq N_1$ then, $\forall i\in\{1,\ldots,m\}$,
$$d_{u_n}(c_n(t_i),c_\infty(t_{i}))<k_2d_{\mathbb{H}^2}(c_n(t_i),c_\infty(t_{i}))<\frac{\epsilon}{3m}~, $$
i.e.
$\sum_{i=1}^m d_{u_n}(c_n(t_i),c_\infty(t_{i}))<\frac{\epsilon}{3}~, $
and triangle inequality gives
$$
\sum_{i=1}^m d_{u_n}(c_\infty (t_i),c_\infty (t_{i+1}))\leq \sum_{i=1}^m d_{u_n}(c_n(t_i),c_n(t_{i+1})) + \frac{2}{3}\epsilon~.  $$
Taking the sup for the decomposition:
$$L_{d_{u_n}}(c_\infty) \leq L_{d_{u_n}}(c_n) + \frac{2}{3}\epsilon~. $$

On the other hand, by Proposition~\ref{prop:conv leng} and by hypothesis, after taking a suitable subsequence, there is $N_2$ such that for $n\geq N_2$,
$$\mathfrak{L}_{u}(c_\infty) \leq L_{d_{u_n}}(c_\infty) + \frac{\epsilon}{3}~. $$

Hence, for $n\geq \mbox{max}\{N_1,N_2\}$,
$\mathfrak{L}_{u}(c_\infty) \leq L_{d_{u_n}}(c_n) + \epsilon~, $
so
$\mathfrak{L}_{u}(c_\infty) \leq \liminf_nL_{d_{u_n}}(c_n) + \epsilon  $
and as $\epsilon$ is arbitrary,
$$\mathfrak{L}_{u}(c_\infty) \leq \liminf_nL_{d_{u_n}}(c_n)~.$$

As $L_{d_{u_n}}(c_n)=d_{u_n}(x,y)$ and $d_u(x,y) \leq \mathfrak{L}_{u}(c_\infty)$, we obtain
$$
d_u(x,y)\leq\liminf_{n\rightarrow\infty}d_{u_n}(x,y)~.
$$

This, together with Lemma~\ref{lem:limsup}, leads to the pointwise convergence of $d_{u_n}$ to $d_u$. The local uniform convergence comes from Lemma~\ref{lem alex trick}.
\end{proof}

\subsection{Convergence on the quotient}

This section is a straightforward adaptation of \cite{fiv}.

The topology on the space of representations $\rho:\pi_1S\to O_0(2,1)\cong \mathrm{PSL}(2,\R)$ can be defined as follows. Let us choose a set of $2g$ generators $(\gamma_1,\ldots,\gamma_{2g})$  of $\pi_1S$, where $g$ is the genus of $S$.
One says that $(\rho_n)_n$ converges to $\rho$ if
$(\rho_n(\gamma_1),\ldots,\rho_n(\gamma_{2g}))$
converges to $(\rho(\gamma_1),\ldots,\rho(\gamma_{2g}))$
in PSL($2,\R)^{2g}\subset \left(\R^4 \right)^{2g}$.
See  e.g. \cite[10.3]{primer} for more details.

\begin{definition}\label{def norm}
A sequence of Fuchsian representations $\rho_{n}$ is \emph{normalized} if the following occurs. We fix a set $\gamma_1,\ldots,\gamma_{2g}$ of generators of $\pi_1S$, and three distinct points
$a,b,c$ of $\partial_\infty \H^2$. Then we ask $a$ to be the attractive fixed point of $\rho_n(\gamma_{2g})$, $b$ to be the repulsive fixed point of $\rho_n(\gamma_{2g})$, and $c$ to be the attractive fixed point of $\rho_n(\gamma_{2g-1})$ for all $n$.
\end{definition}
These are the Fricke coordinates of the Teichm\"uller space.
We will need the following classical result.

\begin{lemma}\label{lem: nielsen}
Let $(\rho_n)_n$ be a sequence of normalized Fuchsian representations. There exist
homeomorphisms $\widetilde\varphi_n : \H^2 \to \H^2$ that satisfy,
for any $\gamma \in \pi_1 S$,
\begin{equation}\label{eq:lift eq} \widetilde\varphi_n\circ \rho(\gamma)= \rho_n(\gamma) \circ \widetilde\varphi_n   \end{equation}
and $(\tilde{\varphi}_n)_n$ converge to the identity map, uniformly on compact sets.
\end{lemma}

\begin{proof}
We know that $(\H^2/\rho_n(\pi_1S))_n$ converges in Teichm\"uller space.
By a theorem of Teichm\"uller, there  are  $K(n)$-quasiconformal homeomorphism from $\H^2/\rho_n(\pi_1S)$ to $\H^2/\rho(\pi_1S)$
with $K(n)\to 1$ \cite{IT}.
Their lifts satisfy \eqref{eq:lift eq}.
Moreover, due to the normalization we chose for the groups $\rho_n(\pi_1S)$, the lifts fix three distinct points on the boundary at infinity of $\H^2$.
Under this normalization condition, up to extract a subsequence,
the sequence of homeomorphisms converge uniformly
\cite[Theorem~1 p.~32]{ahlfors}, \cite{IT}.
\end{proof}

\begin{definition}
A \emph{Fuchsian H-convex function} is a pair $(u,\rho)$, where
$u$ is a H-convex function and $\rho$ is a Fuchsian representation of $\pi_1S$ into $O_0(2,1)$, such that for any  $\gamma\in \pi_1S$, \begin{equation} \label{eq:un inv}u\circ\rho(\gamma)=u~.\end{equation}
We say that a sequence of  Fuchsian H-convex functions $(u_n,\rho_n)_n$
converges to a pair $(u,\rho)$, if $u$ is a H-convex function, $\rho$ a Fuchsian representation,
$(u_n)_n$ converges to $u$ and $(\rho_n)_n$ converges to $\rho$.
\end{definition}

Note that if $(u,\rho)$ is a Fuchsian H-convex function, then
$\rho(\pi_1S)$ acts by isometries on $d_u$. Moreover, if $d_u$ is a distance, then the quotient gives a distance $\bar d_u$ on the compact surface $\H^2/\rho(\pi_1S)$. As the latter is compact, we obtain the following.

\begin{fact}
If $(u,\rho)$ is a Fuchsian H-convex function and $d_u$ a distance, then $d_u$ is a complete distance on $\H^2$.
\end{fact}

\begin{lemma}\label{lem : well-def}
Let $(u_n,\rho_n)_n$ be a sequence of Fuchsian H-convex functions that converges to a pair $(u,\rho)$.
Then $(u,\rho)$ is a Fuchsian H-convex function.
\end{lemma}
\begin{proof}
Let $y\in \H^2$ and $\gamma\in\pi_1S$.
Then, for every $\epsilon>0$
\begin{align}
|u(\rho(\gamma) y)-u(y)|
&\leq |u(\rho(\gamma) y)-u_n(\rho(\gamma) y)|\label{1}
\\
&+|u_n(\rho(\gamma) y)-u_n(\rho_n(\gamma) y)|\label{2}
\\
&+|u_n(\rho_n(\gamma) y)-u_n(y)|\label{3}
\\
&+|u_n(y)-u(y)|<\epsilon\label{4}.
\end{align}
In fact, for $n$ large enough, $d_{\H^2}(\rho(\gamma) y,\rho_n(\gamma) y)\to0$ as $n\to\infty$, and as the $u_n$ are equi-Lipschitz on a sufficiently large compact set (Lemma~\ref{lem:equilip}),
 the absolute value at line \eqref{2} is smaller than $\epsilon/4$ for $n$ large enough. Moreover the absolute value at line \eqref{3} is zero for every $n$ by the $\rho_n(\pi_1S)$-invariance of $u_n$, and the absolute value at lines \eqref{1} and \eqref{4} are smaller then $\epsilon/4$ for $n$ large enough by the uniform convergence of the $u_n$. Since $\epsilon>0$ is arbitrary, this completes the proof.
\end{proof}

\begin{corollary}\label{lem: conv en haut}
Let $(u_n,\rho_n)$ be Fuchsian H-convex functions such that
\begin{itemize}
\item $(u_n,\rho_n)_n$ converges to a pair $(u,\rho)$;
\item there exist $\alpha,\beta >0$ with $\alpha < u_n <\beta$;
\item $d_{u_n}$ are distances with Lipschitz shortest paths;
\item $d_{u_n}$ converge to $d_u$, uniformly on compact sets.
\end{itemize}

Then, on any compact set of $\H^2$,
$d_{u_n}(\widetilde\varphi_n(\cdot),\widetilde\varphi_n(\cdot))$ uniformly converge to $d_u$, where $\widetilde\varphi$ is given by Lemma~\ref{lem: nielsen}.
\end{corollary}
\begin{proof}
By Lemma~\ref{lem: nielsen} and  Corollary~\ref{lem:upper hyp dist},
$x\mapsto d_{u_n}(\tilde{\varphi}_n(x),x)$ uniformly converges to $0$ (on the given compact set).
By the triangle inequality,
$$d_{u_n}(\widetilde\varphi_n(x),\widetilde\varphi_n(y))-d_u (x,y) \leq
d_{u_n}(\widetilde\varphi_n(x),x) + d_{u_n}(\widetilde\varphi_n(y),y)+ d_{u_n}(x,y)-d_u (x,y) $$
and the right-hand side is uniformly less than any $\epsilon >0$ for $n$ sufficiently large by the
preceding argument and the assumption.
On the other hand,
$$d_u(x,y)-d_{u_n}(\widetilde\varphi_n(x),\widetilde\varphi_n(y))
= d_u(x,y)-d_{u_n}(x,y)+d_{u_n}(x,y)-d_{u_n}(\widetilde\varphi_n(x),\widetilde\varphi_n(y))~. $$
We know that $d_u(x,y)-d_{u_n}(x,y)$ is uniformly less that any $\epsilon >0$ for $n$ sufficiently large, and the term $d_{u_n}(x,y)-d_{u_n}(\widetilde\varphi_n(x),\widetilde\varphi_n(y)) $ is less than
$$d_{u_n}(x,\widetilde\varphi_n(x)) + d_{u_n}(y,\widetilde\varphi_n(y))+d_{u_n}(\widetilde\varphi_n(x),\widetilde\varphi_n(y))- d_{u_n}(\widetilde\varphi_n(x),\widetilde\varphi_n(y))$$
that is also uniformly less that any $\epsilon >0$ for $n$ sufficiently large.
\end{proof}

 \begin{proposition}\label{prop:con mink2}
Under the assumptions of Corollary~\ref{lem: conv en haut},
 up to extract a subsequence, $(\H^2/\rho_n(\pi_1S),\bar d_{u_n})$ uniformly converge to  $(\H^2/\rho(\pi_1S), \bar d_u)$.
\end{proposition}

\begin{proof}

Every sequence of representation can be normalized by composing on the left by hyperbolic isometries. Also, applying an isometry to a surface, we do not change its induced distance.
 Hence, in the statement of Proposition~\ref{prop:con mink2}, we may assume that the sequence of representations is normalized.

The maps $\widetilde\varphi_n$ from Lemma~\ref{lem: nielsen} induces homeomorphisms
$$\varphi_n : \H^2/\rho(\pi_1S)\to \H^2/\rho_n(\pi_1S)~.$$

Let $C$ be a compact set of $\H^2$ such that, for any
$p,q\in \H^2/\rho(\pi_1S)$, there are points  $x,y$ in $C$ which are lifts  of $p,q$ respectively, such that \begin{equation}\label{d=d}\bar d_u(p,q)=d_{u}(x,y)~.\end{equation} By definition of $\tilde{\varphi}_n$,  $\tilde{\varphi}_n(x)$ and $\tilde{\varphi}_n(y)$ are lifts of $\varphi_n(p)$ and $\varphi_n(q)$. In  particular,
$d_{u_n}(\tilde{\varphi}_n(x),\tilde{\varphi}_n(y))\geq  \bar d_{u_n} (\varphi_n(p),\varphi_n(q))$,
and
\begin{equation}\label{inequalite 1}
\bar d_{u_n}(\varphi_n(p),\varphi_n(q))- \bar d_u (p,q)  \leq  d_{u_n}(\tilde{\varphi}_n(x),\tilde{\varphi}_n(y))- d_u(x,y)~.
\end{equation}

Now let us look at $\bar d_u (p,q)- \bar d_{u_n}(\varphi_n(p),\varphi_n(q)$.
For any $\gamma\in \pi_1S$,
$$d_u(x,\rho(\gamma)y) \geq d_u(x,y) $$
and by Corollary~\ref{lem: conv en haut}, for any $\epsilon >0$, if $n$ is sufficiently large, uniformly on $C$,

$$d_{u_n}(\tilde{\varphi}_n(x),\rho_n(\gamma)\tilde\varphi_n(y))  + \epsilon\geq d_{u_n}(\tilde\varphi_n(x),\tilde\varphi_n(y)) $$
and as $\bar d_{u_n}(\varphi_n(p),\varphi_n(q))$ is the minimum over $\pi_1S$
of the $d_{u_n}(\tilde{\varphi}_n(x),\rho_n(\gamma)\tilde\varphi_n(y)) $, we obtain
$$\bar d_{u_n}(\varphi_n(p),\varphi_n(q))  + \epsilon\geq d_{u_n}(\tilde\varphi_n(x),\tilde\varphi_n(y))~. $$

This last equation and \eqref{d=d} give
\begin{equation}\label{inequalite 2}
 \bar d_u (p,q)-\bar d_{u_n}(\varphi_n(p),\varphi_n(q)) \leq d_u(x,y)-d_{u_n}(\tilde{\varphi}_n(x),\tilde{\varphi}_n(y))   + \epsilon~.
\end{equation}

From \eqref{inequalite 1}, \eqref{inequalite 2},
and  Corollary~\ref{lem: conv en haut} applied for the compact set $C$,
$$| \bar d_{u_n}(\varphi_n(p),\varphi_n(q))-\bar d_u (p,q)|$$
is uniformly less than any positive number for $n$ sufficiently large.
\end{proof}

\subsection{The polyhedral case}

\begin{definition}
A Fuchsian H-convex function $(u,\rho)$ is \emph{polyhedral}
if the graph $G_u$ is the boundary of the convex hull of the orbit for $\rho(\pi_1S)$ of a
finite number of points in $I^+(0)$.
\end{definition}

\begin{lemma} \label{lemma321}
Let $(u,\rho)$ be a  polyhedral Fuchsian H-convex function.
Then $d_u$ is a complete distance, with Lipschitz shortest paths. Moreover, on the set of Lipschitz paths, $\mathfrak{L}_u=L_{d_u}$.
\end{lemma}
\begin{proof}

It can be   easily showed that $G_u$ is a locally finite gluing of compact convex Euclidean polygons  \cite{Fil12}. In particular,
$d_u$ is a distance. Moreover, the quotient of $G_u$ by $\rho(\pi_1S)$ is compact, hence
$d_u$ is a complete metric.
By Hopf--Rinow theorem, there is a shortest path between each pair of points. Clearly, a shortest path on $G_u$ is a broken line, in particular, it is Lipschitz.

Let us consider a Lipschitz path on $G_u$. It is the union of
a finite number of Lipschitz curves, each contained in a face of $G_u$. Those ones are compact convex polygons in Euclidean planes, where the Euclidean structure is the restriction of the Minkowski metric to the plane containing the face.
It is well-known that in the Euclidean plane, the two ways of measuring length of Lipschitz curves coincide (see e.g. \cite{bur}
for a clear review of this fact and generalizations).

\end{proof}

Putting together Lemma~\ref{lemma321}, Proposition~\ref{prop:con mink} and Proposition~\ref{prop:con mink2}, one obtains the following.

\begin{proposition}\label{prop:main 3}
Let $(u_n,\rho_n)_n$ be a sequence of polyhedral Fuchsian H-convex functions. If
\begin{itemize}
\item there are $\alpha,\beta >0$ with $\alpha < u_n <\beta$,
\item $(\rho_n)_n$ converge to a Fuchsian representation $\rho$,
\end{itemize}
then, up to extract subsequences, there is a Fuchsian H-convex function $(u,\rho)$
such that
$(\H^2/\rho_n(\pi_1S),\bar d_{u_n})_n$ uniformly converge to
$(\H^2/\rho(\pi_1S),\bar d_{u})$.
\end{proposition}

\section{Compactness results} \label{sec 4}

The aim of this section is to give, in some sense, the converse to Proposition~\ref{prop:main 3}.
More precisely, we want to prove the following.

\begin{proposition}\label{prop finale}
Let $(S,d)$ be a metric of non-positive curvature, and let $(u_n,\rho_n)$ be polyhedral Fuchsian  H-convex functions  such that
\begin{itemize}
\item the representations $\rho_n$ are normalized in the sense of Definition~\ref{def norm}.
\item the sequence of compact surfaces $(\H^2/\rho_n(\pi_1S),\bar d_{u_n})$ uniformly converges to $(S,d)$~;
\end{itemize}
then, up to extract a subsequence,
\begin{itemize}
\item there exist $\alpha,\beta>0$ such that
$$\alpha < u_n <\beta$$
\item the sequence $\rho_n$ converge to a Fuchsian representation $\rho$.
\end{itemize}
\end{proposition}

Under the hypothesis of Proposition~\ref{prop finale},
 there are homeomorphisms
$$\psi_n:\H^2/\rho_n(\pi_1S) \to S $$
such that, if we denote by $d_n$ the push-forward by $\psi_n$ of  $(\H^2/\rho_n(\pi_1S),d_{u_n})$, then on $S$, the sequence $(d_n)_n$ uniformly converge to $d$.
By equivariance, if $\tilde d_n$ is the lift of $d_n$ to the universal cover of $S$, and $\tilde\psi_{n}$ is a lift of $\psi_{n}$,
for any $x,y\in \H^2$,
\begin{equation}
\label{eq:equivariance2}
d_{u_n}(x,\rho_n(\gamma)y)=\tilde d_n (\tilde\psi_{n}(x), \gamma.\tilde\psi_{n}(y))~.
\end{equation}

\subsection{Lower bound}\label{section:radial-functions}

In this Paragraph we will show that there is a uniform lower bound for the functions $u_{n}$, $n\in\mathbb{N}$.
First, with \eqref{eq:equivariance2}, Corollary~\ref{lem:min-length-gamma>K-for-all-n} translates as follows.

\begin{corollary}
 There exists $G>0$ and $N>0$ such that, for every element $\gamma\in\pi_{1}S\setminus{\{0\}}$,  for any $n>N$ and for any $x\in\H^ 2$:
\begin{equation}
\label{ineqality:min-length-gamma>K-for-all-n}
d_{u_n}(x,\rho_n(\gamma)x)\geq G~.
\end{equation}
\end{corollary}

We also need the following lemma for the hyperbolic metrics.

\begin{lemma}\label{lemma:buser-short-geodesic}
There exists $R>0$ (depending only on the genus of $S$) such that, for any $n\in \N$, for any $x\in \H^2$, there exists $\gamma_n \in \pi_1S\setminus \{0\}$, such that
$$d_{\H^2}(x,\rho_n(\gamma_n)x) < R~. $$
\end{lemma}
\begin{proof}
Suppose the converse, i.e., for $R$ arbitrary large there is however a point $x$ and an index $n$ such that for any $\gamma\in\pi_{1}S\setminus\{0\}$ the inequality
\begin{equation}
\label{ineq:converse-inequality}
d_{\H^2}(x,\rho_n(\gamma)x) \geq R
\end{equation}
holds. Let us consider the Dirichlet polygon $D$ for
$\rho_n(\pi_{1}S)$ centred at $x$. Let $r$ be the radius of the largest disc centred at $x$ contained in $D$. Then there is $\gamma_x$ such that $r= d_{\H^2}(x,\rho_n(\gamma_x)x)/2$. By~(\ref{ineq:converse-inequality}) we have that $r\geq R/2$. Therefore, the area of the disc would be arbitrarily large, that is impossible, as it is less than the area of $D$, and this one  depends only on the genus of $S$.
\end{proof}

\begin{proposition}\label{lem:u_n>alpha}
Under the hypothesis of Proposition~\ref{prop finale},
 there exists  $\alpha>0$ such that,  for any $n$ up to extract a subsequence and for any $x\in \H^2$,
\begin{equation*}
\label{ineqality:u_n>alpha}
u_{n}(x)>\alpha~.
\end{equation*}
\end{proposition}
\begin{proof}
Let us suppose the converse, i.e. that for an arbitrary $\epsilon > 0$, there exists $n$ and $x_n\in\H^2$  with
\begin{equation}\label{eq:epsilon un}u_n(x_n)<\epsilon~.\end{equation}
Let the element $\gamma_n\in \pi_1S$ from Lemma~\ref{lemma:buser-short-geodesic} such that
\begin{equation}\label{Mg}d_{\H^2}(x_n,\rho_n(\gamma_n)x_n)\leq R~. \end{equation}

Let $\sigma$ be the intersection of the graph of $u_n$ with the Minkowski $2$-plane which passes through the points $x_n$, $\rho_{n}(\gamma_{n})x_{n}$, and the origin of the coordinate system in $\R^{2,1}$; also, let $L(\sigma)$ be its length. In particular,
\begin{equation}
\label{ineqality:distance<length-curve-sigma-n}
d_{u_n}(x_{n},\rho_{n}(\gamma_{n}) x_{n})\leq L(\sigma)~.
\end{equation}

 By construction, $\sigma$ is a piecewise linear spacelike curve in a Minkowski $2$-plane. Applying several times the inverse triangle inequality for Minkowski $2$-plane (see \eqref{reversed ti}) and using \eqref{eq:un inv}, we get
\begin{equation}
\label{ineqality:sigma-n-inverse-triangle}
L(\sigma)\leq\|u_n(x_n)\rho_{n}(\gamma_{n}) x_{n}-u_n(x_n)x_{n}\|_{-}~.
\end{equation}

Looking at Figure~\ref{fig1}, we write down:
\begin{equation*}
\label{ineqality:estimate-mink-by-hyp-n}
\|u_n(x_n)\rho_{n}(\gamma_{n}) x_{n}-u_n(x_n)x_{n}\|_{-}=\sqrt{2} u_n(x_n) (\cosh (d_{\H^2}(x_n,\rho_n(\gamma_n)x_n))) -1)^{1/2}~.
\end{equation*}
So with \eqref{ineqality:distance<length-curve-sigma-n}, \eqref{ineqality:sigma-n-inverse-triangle}, \eqref{ineqality:min-length-gamma>K-for-all-n}, \eqref{eq:epsilon un} and \eqref{Mg},
\begin{equation*}
\label{ineqality:combined-ineq-n}
G \leq\sqrt{2} \epsilon (\cosh (R) -1)^{1/2}~
\end{equation*}
that is impossible as $\epsilon$ is arbitrarily small.
\end{proof}

\subsection{Convergence of groups} \label{paragraph:convergence-of-groups}

Here we prove the main step in the proof of Theorem~\ref{thm:main 2}.
We adapt an argument that was developed by J.-M.~Schlenker in \cite{Schlenker2004} for convex surfaces in de Sitter space.

\begin{proposition}\label{cor gp cv}
Under the hypothesis of Proposition~\ref{prop finale},
up to extract a subsequence, the sequence  $(\rho_n)_n$ converges
to a Fuchsian representation $\rho$.
\end{proposition}

Let us denote by $P_n$ the graph of $u_n$, and by $f_n$ the restriction to $P_n$
of the squared distance from the origin in Minkowski space, i.e. for each $y\in P_n$ we define
\begin{equation}\label{eq:def-function-f-P}
f_n(y)=-\langle y,y\rangle_{-}~.
\end{equation}
The function $f_n$ is positive and invariant under the action of
$\rho_n(\pi_1S)$, which acts cocompactly on $P_n$, hence it attains
its extremal values.

\begin{fact}\label{lemma:positive-jump}
Let $c$ be an arc-length parametrized shortest path on $P_n$.
\begin{enumerate}
\item  $(f_n\circ c)'$ has a positive  jump at its singular points.
\item At non-singular points,  $(f_n\circ c)''=-2$ and $(f_n\circ c)^{(n)}=0$ for $n\geq 3$.
\item $f_n\circ c$ is regular at local maxima.
\end{enumerate}
\end{fact}
\begin{proof}
\begin{enumerate}
\item The first assertion follows because the singular point of
$f_n\circ c$ are when $c$ crosses an edge, and the result follows by convexity of $P_n$ (note that clearly the intersection of a shortest path with
an edge reduces to a single point).
\item If the point $c(0)$ is not singular, then locally
$c$ is a spacelike segment that can be written as  $c(0)+tv$, $t\in (-\epsilon, \epsilon)$, and $v$ a unit spacelike segment.
So $$f_n(c(t))=-t^2-2\langle c(0),v\rangle_{-}t-\langle c(0),c(0)\rangle_{-}~.$$
\item
If $f_n$ has a local maximum at $t_{M}$, then there is a small neighbourhood $(a,b)$ of $t_{M}$ such that the function $f_n$ is regular on $(a,t_{M})\cup(t_{M},b)$, $f_n(c(t))'\geq0$ for all $t\in(a,t_{M})$, and $f_n(c(t))'\leq0$ for all $t\in(t_{M},b)$. Thus, once we assume that $t_{M}$ is a singular point of the application $f_n$, then, by the first fact, $(f_n\circ c)'$ has a positive jump at $t_{M}$, and so there are points on $(t_{M},b)$ where $(f_n\circ c)'$ is strictly greater than $0$ which leads to a contradiction.
\end{enumerate}
\end{proof}

Let $y_M(n)$ be a point on $P_n$  where $f_n$ attains its maximum. Let $x_M(n)$ be the image of $y_M(n)$ by the radial projection from $I^+(0)$ onto $\H^2$.
Among all the maxima of $f_n$, $y_M(n)$ is chosen
as follows.
As $(d_n)$ uniformly converge to $d$, by  Lemma~\ref{sup metrique}, there is a metric
$d_{\sup}$ majorizing the $d_n$. We choose
 $y_M(n)$  such that all the $\tilde\psi_{n}(x_M(n))$ belong to the same Dirichlet fundamental region for $\tilde d_{\sup}$.

Let $\gamma\in \pi_1S$. Let $c:[0,l_\gamma(n)]\to P_n$ be an arc-length parametrized shortest path
of $P_n$ between  $y_M(n)$ and $\rho_n(\gamma)y_M(n)$.

\begin{lemma}\label{lemma:estimate-fP}

With the notation above,
\begin{equation}\label{eq:estimate-fP}
\int_{0}^{l_\gamma(n)} \big{|}\langle c,c'\rangle_{-}\big{|} \leq \frac{l_\gamma(n)^2}{2}~.
\end{equation}
\end{lemma}

\begin{proof}
To simplify the notation, let us denote $f_n\circ c$ by $f$ in this proof.
Let $T$ be a point where $f$ attains a local maximum. By Fact~\ref{lemma:positive-jump}, $(f(T))'=0$, hence
the Taylor polynomial of the function $f$ at $T$ is
\begin{equation*}\label{eq:taylor-t-mu}
P^{f}_{T}(t)=f(T)-(t-T)^2~.
\end{equation*}

As the function $f$ can have positive jumps, $f(t)\geq P^{f}_{T}(t)$ for all $t$. Let the points $t_{-}$ and $t_{+}$ be the closest points to $T$, $t^{-}<T<t_{+}$, where $f$ attains local minima, in particular $f(t_{-})<f(T)$ and $f(t_{+})<f(T)$. So we have that
\begin{equation*}\label{eq:taylor-t+--1}
f(t_{-})\geq P^{f}_{T}(t_{-})=f(T)-(t_{-}-T)^2\quad\mbox{and}\quad f(t_{+})\geq P^{f}_{T}(t_{+})=f(T)-(t_{+}-T)^2~,
\end{equation*}
or, in other words,
\begin{equation}\label{eq:taylor-t+--2}
0<f(T)-f(t_{-})\leq (T-t_{-})^2\quad\mbox{and}\quad 0<f(T)-f(t_{+})\leq (t_{+}-T)^2~.
\end{equation}

By definition, $f$ has global (and also local) maxima at $t=0$ and $t=l_{\gamma}(n)$. Let us now introduce the decomposition $0=T_{0}<t_{1}<T_{1}<\ldots<t_{N}<T_{N}=l_{\gamma}(n)$ of the segment $[0,l_{\gamma}(n)]$, where $f$ attains local minima at $t_{i}$, $i=1,\ldots,N$, and local maxima at $T_{i}$, $i=0,\ldots,N$. Hence, at the points of differentiability of $f$, we have that $f'(t)\geq0$ for all intervals $[t_{i},T_{i}]$, and that $f'(t)\leq0$ for all intervals $[T_{i},t_{i+1}]$. Note also that $f'(t)=-2\langle c(t),c'(t)\rangle_{-}$ at all points where $f$ is differentiable. Therefore, by (\ref{eq:taylor-t+--2}), for all intervals $[t_{i},T_{i}]$ and $[T_{i},t_{i+1}]$ we get
\begin{equation}\label{eq:int-min-max}
\int_{t_{i}}^{T_{i}} \big{|}\langle c(t),c'(t)\rangle_{-}\big{|} dt=\int_{t_{i}}^{T_{i}}\frac{f'(t)}{2}dt=\frac{f(T_{i})-f(t_{i})}{2}\leq\frac{(T_{i}-t_{i})^2}{2}
\end{equation}
\quad\mbox{and}\quad
\begin{equation}\label{eq:int-max-min}
\int_{T_{i}}^{t_{i+1}} \big{|}\langle c(t),c'(t)\rangle_{-}\big{|} dt=\int_{T_{i}}^{t_{i+1}}\frac{-f'(t)}{2}dt=\frac{f(T_{i})-f(t_{i+1})}{2}\leq\frac{(t_{i+1}-T_{i})^2}{2}~.
\end{equation}

Hence
\begin{equation*}\label{eq:f-integral-estimation1}
\int_{0}^{l_{\gamma}(n)} \big{|}\langle c,c'\rangle_{-}\big{|} =
\sum_{i=1}^{N}\Bigg{[} \int_{T_{i-1}}^{t_{i}} \big{|}\langle c,c'\rangle_{-}\big{|}  + \int_{t_{i}}^{T_{i}} \big{|}\langle c,c'\rangle_{-}\big{|} \Bigg{]}\leq
\frac{1}{2}\sum_{i=1}^{N}\Big{[} (t_{i}-T_{i-1})^2 + (T_{i}-t_{i})^2 \Big{]}
\end{equation*}
and as for non-negative real numbers, $\sum_{j=1}^{m}(a_{j})^2\leq(\sum_{j=1}^{m}a_{j})^2$, we arrive at
\begin{equation*}\label{eq:f-integral-estimation4}
\int_{0}^{l_{\gamma}(n)} \big{|}\langle c,c'\rangle_{-}\big{|}\leq\frac{1}{2}\Bigg{[}\sum_{i=1}^{N}\Big{[} (t_{i}-T_{i-1}) + (T_{i}-t_{i}) \Big{]} \Bigg{]}^2=\frac{[T_{N}-T_{0}]^2}{2}=\frac{l_{\gamma}(n)^2}{2}~.
\end{equation*}
\end{proof}

\begin{corollary}\label{lemma:translation-length-estimate}
We have, for any $n$,
\begin{equation*}
d_{\H^2}(x_M(n),\rho_n(\gamma)x_M(n)) \leq \frac{B_\gamma}{\alpha}+\frac{B_\gamma^2}{2\alpha^2}~,
\end{equation*}
where $\alpha$ is the uniform lower bound of the functions
 $u_n$ given by
Proposition~\ref{ineqality:u_n>alpha}, and $B_\gamma$ is a constant depending only on $\gamma$.
\end{corollary}
\begin{proof}
By Lemma~\ref{eq:principale},
$$
d_{\H^2}(x_M(n),\rho_n(\gamma)x_M(n)) \leq \frac{l_\gamma(n)}{\alpha}+\frac{1}{\alpha^2}\int_{0}^{l_\gamma(n)} \big{|}\langle c,c'\rangle_{-}\big{|} ~.$$

With \eqref{eq:estimate-fP}, we obtain
$$
d_{\H^2}(x_M(n),\rho_n(\gamma)x_M(n)) \leq \frac{l_\gamma(n)}{\alpha}+\frac{l_\gamma(n)^2}{2\alpha^2}~.$$

Now, $l_\gamma(n)=d_{u_n}(y_M(n),\rho_n(\gamma)y_M(n))$
so from \eqref{eq:equivariance2},
 $$l_\gamma(n)=\tilde{d}_n(\tilde\psi(x_M(n)),\gamma.\tilde\psi(x_M(n)))\leq \tilde{d}_{\sup} (\tilde\psi(x_M(n)),\gamma.\tilde\psi(x_M(n)))~.$$
  As by definition, $\tilde\psi(x_M(n))$ are all belonging to a same compact set of $\tilde{S}$, the existence of $B_\gamma$ follows

\end{proof}

For $\gamma\in \pi_1S$, let us denote by $L_{\rho_n}(\gamma)$ the length of the geodesic representative of $\gamma$ in $\H^2/\rho_n(\pi_1S)$. In particular,
$L_{\rho_n}(\gamma)=\min_{x\in \H^2}d_{\H^2}(x,\rho_n(\gamma)x)$.


\begin{corollary} \label{lemma-bounded-closed curves}
Let $\gamma\in\pi_1S$.
Then, there exists $B_\gamma>0$ such that for
any $n$,
\begin{equation}\label{eq:translation-length-estimate}
L_{\rho_n}(\gamma) \leq \frac{B_\gamma}{\alpha}+\frac{B_\gamma^2}{2\alpha^2}~.
\end{equation}
\end{corollary}

We have proved that for every $\gamma\in\pi_1S$,
$L_{\rho_n}(\gamma)$ is bounded from above by a constant which does not depend on $n$.
 We can now use the following result.

\begin{proposition}[Proposition~7.11 in \cite{FLP2012}]\label{lemma-teichmuller-curves}
There exist  $\gamma(i)\in \pi_1S$, $i=1,\cdots, 9g-9$, such
that the map
from Teichm\"uller space of $S$ to $\R^{9g-9}$ that gives
the length of the geodesic representatives of  $\gamma(i)$
is injective and proper.
\end{proposition}

Corollary~\ref{lemma-bounded-closed curves} and Proposition~\ref{lemma-teichmuller-curves} imply that the sequence of elements of Teichm\"uller space
defined by $(\rho_n(\pi_1S))_n$ is lying in a compact set, so  up to extract a subsequence,
the sequence $(\rho_n(\pi_1S))_n$ converges in Teichm\"uller space.
By the normalization we defined in Definition~\ref{def norm} for the groups $\rho_n(\pi_1S)$, Proposition~\ref{cor gp cv} follows.

\subsection{Upper bound}

\begin{proposition}\label{lem upper bound}
Under the hypothesis of Proposition~\ref{prop finale},
 there exists  $\beta>0$ such that for any $n$ and for any $x\in \H^2$,
\begin{equation*}
u_{n}(x)<\beta~.
\end{equation*}
\end{proposition}
\begin{proof}

As the sequence of representations converges (Proposition~\ref{cor gp cv}), there exists a compact set
$C\subset \H^2$ which contains a fundamental domain for $\rho_n(\pi_1S)$ for all $n$.
For each $n$, let $x_n$ be a point of $C$ which realizes the minimum
of $u_n$.
Suppose that $(u_n(x_n))_n$ is not bounded from above.
In particular, up to extract a subsequence, $u_n$ are uniformly bounded from below by a positive constant $\alpha$, and moreover
one can take
$\alpha$ arbitrarily large in Corollary~\ref{lemma-bounded-closed curves}.
But then, for any $\gamma\in \pi_1S$, $L_{\rho_n}(\gamma)$ is arbitrarily small, that is impossible on a compact hyperbolic surface. For example, one can use the fact that  only finitely many closed geodesics have length less than a given constant
\cite[Theorem~1.6.11]{Bus}.
Hence  there exists $b$ such that for any $n$, $u_n(x_n)<b$.

From Proposition~\ref{lem:u_n>alpha}, $\alpha \leq u_n(x_n)$ independently of $n$. Let $y\in C$. As $C$ is compact, there exists $M$ such that, for any $n$,
$$-\langle x_n,y\rangle_- = \cosh d_{\H^2}(x_n,y) \leq M~. $$
As the surfaces defined by $u_n$ are spacelike, by  Lemma~\ref{lem:edge spacelike} below,
$$0<\langle u_n(x_n)x_n-u_n(y)y,u_n(x_n)x_n-u_n(y)y\rangle_-~. $$

Developing the right-hand side and using the bounds introduced above, we arrive at the condition
$$0<-u_n(y)^2 + Mb u_n(y) -\alpha^2~, $$
where the constants are independent of the choice of $n$ and $y$ in $C$. So the $u_n$ are uniformly bounded from above.

\end{proof}

\begin{lemma}\label{lem:edge spacelike}
For any $x,y$ on a spacelike convex surface, $x\not= y$, the segment between $x$ and $y$ is spacelike.
\end{lemma}
\begin{proof}
Assume the contrary, i.e. the segment $[x,y]$ is lightlike or timelike.
Consider the intersection of the surface with the plane  passing through the origin and $x,y$. It is a convex curve,
which has a line parallel to $[x,y]$ as support line, hence there is a support plane of the surface containing this line, and this plane cannot be spacelike if this line is not spacelike, that is a contradiction.
\end{proof}

\section{Proof of Theorem \ref{thm:main 2}}\label{sec 5}

Let us now consider the statement of Theorem~\ref{thm:main 2}.
So let $d$ be a metric with non-positive curvature on the compact surface $S$.
By Corollary~\ref{thm:approx}, there exists a sequence $(d_{n})_{n\in\mathbb{N}}$  of polyhedral metrics with non-positive curvature
on $S$ that converges uniformly to $d$.
By Theorem~\ref{thm:cas poly}, for each $n\in\mathbb{N}$ there
is a Fuchsian convex isometric immersion
$(\phi_n,\rho_n)$, such that $\phi_n(\tilde S)$ is a convex polyhedral surface. Up to compose by global Minkowski isometries, we consider that the sequence of representations is normalized in the sense of Definition~\ref{def norm}.

 Then, if $\phi_n(\tilde S)$ is the graph of the polyhedral Fuchsian H-convex function $(u_n,\rho_n)$, Proposition~\ref{prop finale} applies: there is a subsequence of $\rho_n$ converging to a Fuchsian representation $\rho$, and $\alpha,\beta>0$ such that $\alpha < u_n <\beta$.

So Proposition~\ref{prop:main 3} applies: there is a function $u$ such that the induced distance on the quotient of $d_u$ by $\rho(\pi_1S)$, say $(S,m)$, is the uniform limit of the $(S,d_n)$.
The limit for uniform convergence (actually for the weaker Gromov-Hausdorff convergence) is unique (up to isometries)
 \cite{BBI2001}, hence $(S,m)$ is isometric to $(S,d)$. Theorem~\ref{thm:main 2} is proved.

\begin{spacing}{0.9}
\begin{footnotesize}
\bibliography{minkowski}
\bibliographystyle{alpha}
\end{footnotesize}
\end{spacing}

\end{document}